\documentclass[11pt,a4page]{article}
\usepackage{amssymb}
\usepackage{amsmath}
\usepackage{amsfonts}
\usepackage{appendix}
\usepackage{amsthm}
\usepackage{bookmark}
\usepackage{cite}
\usepackage{changes}
\usepackage{color}
\usepackage{verbatim}

\date{}

\newtheorem{thm}{Theorem}[section]

\newtheorem{lem}[thm]{Lemma}
\newtheorem{prop}[thm]{Proposition}

\theoremstyle{definition}
\newtheorem{defn}{Definition}[section]

\theoremstyle{remark}

\numberwithin{equation}{section}

\DeclareMathSymbol{\C}{\mathalpha}{AMSb}{"43}

\numberwithin{equation}{section} \theoremstyle{plain}

\renewcommand{\theequation}{\arabic{section}.\arabic{equation}}

\def\R{{\textbf{R}}}
\def\H{\mathcal{H}}
\def\A{\mathcal{A}}

\def\f{\frac}
\def\R{\mathbb R}
\def\C{\mathbb C}

\def\Om{\Omega}
\def\var{\varphi}
\def\ep{\epsilon}
\def\ds{\displaystyle}

\title {Mass concentration and uniqueness of ground states for  mass subcritical rotational nonlinear Schr\"{o}dinger equations}
\date{\today}
\author{ Yongshuai Gao\thanks{Email: ysgao@mails.ccnu.edu.cn.}, and\, Yong Luo\thanks{Email: yluo@ccnu.edu.cn. Y. Luo is partially supported by the Project funded by China Postdoctoral Science Foundation No. 2019M662680.
	} 
\\
\small \it	School of Mathematics and Statistics,\\
\small \it Central China Normal University,\\
\small \it Wuhan 430079, People's Republic of China\\}

\begin{document}
\baselineskip= 15pt
\maketitle
\begin{abstract}
This paper considers ground states of  mass subcritical rotational nonlinear Schr\"{o}dinger equation
\begin{equation*}
-\Delta u+V(x)u+i\Om(x^\perp\cdot\nabla u)=\mu u+\rho^{p-1}|u|^{p-1}u \,\ \text{in} \,\ \R^2,
\end{equation*}
where $V(x)$ is an external potential, $\Om>0$ characterizes the rotational velocity of the trap $V(x)$, $1<p<3$ and $\rho>0$ describes the strength of the attractive interactions. It is shown that ground states of the above equation can be described equivalently by minimizers of the $L^2-$ constrained variational problem.
We prove that minimizers exist for any $\rho\in(0,\infty)$ when $0<\Om<\Om^*$, where $0<\Om^*:=\Om^*(V)<\infty$ denotes the critical rotational velocity of $V(x)$. While $\Om>\Om^*$, there admits no minimizers for any $\rho\in(0,\infty)$. For fixed $0<\Om<\Om^*$, by using energy estimates and blow-up analysis, we also analyze the limit behavior of minimizers as $\rho\to\infty$. Finally, we prove that up to a constant phase, there exists a unique minimizer when $\rho>0$ is large enough and $\Om\in(0,\Om^*)$ is fixed.



\end{abstract}
\bigskip

\section{Introduction}
In this paper, we study ground states of the following time-independent nonlinear Schr\"{o}dinger equation
\begin{equation}\label{bc1-1}
-\Delta u+V(x)u+i\Om(x^\perp\cdot\nabla u)=\mu u+\rho^{p-1}|u|^{p-1}u \,\ \text{in} \,\ \R^2,
\end{equation}
where $V(x)$ is a trapping potential, $\Om>0$ characterizes the rotational velocity of the trap $V(x)$, $x^\perp=(-x_2,x_1)$ with $x=(x_1,x_2)\in\R^2$, $\mu\in\R$ is the chemical potential, $\rho>0$ describes the strength of the attractive interactions, and $1<p<3$.
The rotational nonlinear Schr\"{o}dinger equation \eqref{bc1-1} appears in many aspects of physics, such as nonlinear optics, plasma physics and so on \cite{RS,SS}.
In particular, equation \eqref{bc1-1} with $p=3$ is also known as Gross-Pitaevskii equation, which models two dimension attractive Bose-Einstein condensates in a trap $V(x)$ rotating at the velocity $\Om$, see \cite{A,DS,F,S} and the references therein.
	
As illustrated by Theorem \ref{bc} in the Appendix, ground states of \eqref{bc1-1} can be described equivalently by minimizers of the following mass subcritical ($L^2$-subcritical) constraint variational problem
\begin{equation}\label{A1cvp}
I(\rho):=\inf\limits_{\{u\in \H, \|u\|^2_2=1\}}E_\rho(u),
\end{equation}
where the Gross-Pitaevskii (GP) energy functional $E_\rho(u)$ is defined by
\begin{equation}\label{A2ef}
E_\rho(u):=\int_{\R^2}\big(|\nabla u|^{2}+V(x)|u|^{2}\big)dx-\f{2\rho^{p-1}}{p+1}\int_{\R^2}|u|^{p+1}dx-\Om\int_{\R^2}x^\perp\cdot(iu,\nabla u)dx.
\end{equation}
Here $(iu,\nabla u)=i(u\nabla \bar{u}-\bar{u}\nabla u)/2$, and the space $\H$ is defined as
\begin{equation}\label{A3space}
\H:=\Big\{u\in H^1(\R^2,\C):~\int_{\R^2}V(x)|u|^{2}dx<\infty\Big\}
\end{equation}
with the associated norm $\|u\|_\mathcal{H}=\{\int_{\R^2}\big(|\nabla u|^2+(1+V(x))|u|^2\big)dx\}^{\f{1}{2}}$. To discuss equivalently ground states of \eqref{bc1-1}, in this paper, we shall therefore focus on investigating \eqref{A1cvp}, instead of \eqref{bc1-1}.

For the non-rotational case $\Om=0$, it is easy to obtain from the diamagnetic inequality $|\nabla u(x)|\geq |\nabla |u|(x)|$  that minimizers of (\ref{A1cvp}) are essentially real-valued. Many qualitative properties of the
real-valued  minimizers were studied in the last decades, see \cite{C,CL,GZZ1,LZ,L1,M} and the references therein. Indeed, because $I(\rho)$ is a  $L^2-$subcritical problem ($1<p<3$),  one can use the celebrated concentration-compactness lemma \cite{L1} to prove that real-valued minimizers of $I(\rho)$ always exist  for any $\rho\in(0,\infty)$. Moreover, by applying the energy method developed in \cite{GS,GWZZ,GZZ2}, \cite{LZ} also obtained the mass concentration behavior and the local uniqueness of minimizers for $I(\rho)$ as $\rho\to\infty$ recently.

On the other hand, minimizers of $I(\rho)$ are no longer real-valued in the rotational case $\Om>0$.
To our knowledge, the existence of complex-valued minimizers for problem (\ref{A1cvp})  was obtained earlier in \cite{EL}, which however focused mainly on the special case where $V(x)=|x|^2$ and $\Omega=2$. Subsequently, \cite{CE} proved the stability of complex-valued minimizers for $I(\rho)$ in $\R^3$.
Recently, much attention has been attracted to the problem $I(\rho)$ again due to its significance on rotating BEC theory \cite{BC2,GLP,GLY,LNR}.  According to \cite{LNR}, the model of two dimensional rotating BEC can be described by the $L^2-$critical version of $I(\rho)$, for which the exponent $1<p<3$ was replaced by $p=3$ in the  functional $E_\rho(u)$. For this $L^2-$critical problem, \cite{GLY,LNR}
proved that there exist critical constants  $0<\Om^*:=\Om^*(V)\leq\infty$ and  $\rho^*>0$,
such that for any $\Om\in(0,\Om^*$), minimizers of $I(\rho)$ exist if and only if $\rho<\rho^*$. Moreover, by developing the method of inductive symmetry, \cite{GLY} also proved the uniqueness and free-vortex of minimizers for $I(\rho)$ as $\rho\nearrow\rho^*$ for some suitable class of radial potential $V(x)=V(|x|)$. More recently, the authors in \cite{GLP} have generalized the local uniqueness result of \cite{GLY} to the non-radially symmetric case of $V(x)$.

Compared with those analysis  of the real-valued minimizers in \cite{GZZ1,LZ,M}, we remark that there are some new problems need to be solved in the rotational case  $\Om>0$ for $I(\rho)$ (see \cite{GLP,GLY,LNR}). One of the key issue is to obtain a refined estimates of the imaginary part of the complex-valued minimizers, for which some new phase transformations about minimizers and some properties of the linearized operator was fully utilized in \cite{GLP,GLY,LNR}. Moreover, the $L^2-$ critical exponent $p=3$ also plays an important role in their analysis and calculations. Therefore, a natural question to ask is that whether those results in \cite{GLP,GLY}, which focused on studying $I(\rho)$ with $p=3$, still holds true for any $1<p<3\,$? The main purpose of this paper is to settle this problem, {\em we shall investigate the existence and nonexistence, the limit behavior and local uniqueness of complex-valued minimizers of (\ref{A1cvp}) for any $1<p<3$ and $\Omega>0$}.

Throughout the paper, we always assume that the trapping potential  $V(x)$ satisfies
\begin{equation}\label{A5Vdecondition}
0\leq V(x)\in L^\infty_{loc}(\R^2),\,\ \varliminf\limits_{|x|\to\infty}\f{V(x)}{|x|^2}>0,
\end{equation}
and we denote the critical rational speed $\Omega^*$  by
\begin{equation}\label{A6Omdedefintion}
\Om^*:=\sup\Big\{\Om>0: \ V(x)-\f{\Om^2}{4}|x|^2\to\infty \,\ \text{as} \ |x|\to\infty\Big\}.
\end{equation}
To state our main results, we first recall the following Gagliardo-Nirenberg inequality \cite{WMI}
     \begin{equation}\label{A7GN}
     \|u\|^{p+1}_{p+1}\leq \f{\f{p+1}{2}(\f{2}{p-1})^{\f{p-1}{2}}}{\|w\|^{p-1}_2}\|\nabla u\|^{p-1}_2\|u\|^2_2,\,\ u\in H^1(\R^2,\R), \,\ 1<p<3,
     \end{equation}
where $w$ is the unique (up to translations) positive radial solution of the following nonlinear scalar field equation \cite{GNN,K,WMI}
    \begin{equation}\label{A4wdeequation}
    -\Delta u+u-u^p=0 \,\ \text{in} \,\ \R^2, \,\ \text{where} \,\ u\in H^1(\R^2,\R),
    \end{equation}
and the equality in (\ref{A7GN}) is achieved at $u=w$. Note also from \cite [Proposition 4.1] {GNN} that $w=w(|x|)>0$ satisfies
     \begin{equation}\label{A8wdecay}
      w(x), \ |\nabla w(x)|=O(|x|^{-\f{1}{2}}e^{-|x|})\,\ \text{as} \,\ |x|\rightarrow \infty.
     \end{equation}
Moreover, it follows from \cite [Lemma 8.1.2] {C} that $w$ satisfies
     \begin{equation}\label{A9identity}
     \int_{\R^2}|\nabla w|^2dx=\f{p-1}{p+1}\int_{\R^2}w^{p+1}dx=\f{p-1}{2}\int_{\R^2}w^2dx.
     \end{equation}
Recall also from \cite{LL} the following diamagnetic inequality: for $\A=\f{\Om}{2}x^\perp$,
    \begin{equation}\label{A10diamagnetic.inequ}
     |\nabla u|^2-\Om x^\perp\cdot(iu,\nabla u)+\f{\Om^2}{4}|x|^2|u|^2=|(\nabla-i\A)u|^2
     \geq\big|\nabla |u|\big|^2,\,\ u\in H^1(\R^2,\C).
     \end{equation}

Applying the Gagliardo-Nirenberg inequality (\ref{A7GN}) and the diamagnetic inequality (\ref{A10diamagnetic.inequ}), our first result on the existence and nonexistence of minimizers for $I(\rho)$ is stated as the following theorem.
\vskip 0.2truein
    \begin{thm}\label{thm.exist}
Suppose $V(x)$ satisfies (\ref{A5Vdecondition}), then we have
\begin{enumerate}
\item If $0<\Om<\Om^*$, then $I(\rho)$ admits at least one minimizer for any $\rho\in(0,\infty)$;
\item If $\Om>\Om^*$, then $I(\rho)$ admits no minimizer for any $\rho\in(0,\infty)$.
\end{enumerate}
    \end{thm}
By variational theory, any minimizer $u_\rho$ of $I(\rho)$ satisfies the following Euler-Lagrange equation
    \begin{equation}\label{A11udeELE}
     -\Delta u_\rho+V(x)u_\rho+i\Om(x^\perp\cdot\nabla u_\rho)=\mu_\rho u_\rho+\rho^{p-1}|u_\rho|^{p-1}u_\rho \,\ \text{in} \,\ \R^2,
     \end{equation}
where $\mu_\rho\in\R$ is a suitable Lagrange multiplier satisfying
\begin{equation}\label{1.12}
\mu_\rho=I(\rho)-\f{p-1}{p+1}\rho^{p-1}\int_{\R^2}|u_\rho|^{p+1}dx, \,\ \rho\in(0,\infty).
\end{equation}
By making full use of the equation \eqref{A11udeELE}, we next focus on investigating the limit behavior of minimizers for $I(\rho)$ as $\rho\to\infty$, for which we define
\vskip 0.2truein
    \begin{defn}\label{Def.homo}
The nonnegative function $h(x):\R^2\to\R^+$ is homogeneous of degree $s\in\R^+$ (about the origin), if there exists some $s>0$ such that
     \begin{equation*}
     h(tx)=t^sh(x)\,\ \text{in} \,\ \R^2 \,\ \text{for any} \,\ t>0.
     \end{equation*}
    \end{defn}
\noindent This definition implies that if $h(x)\in C(\R^2)$ is homogeneous of degree $s>0$, then
     \begin{equation*}
     0\leq h(x)\leq C|x|^s \,\ \text{in} \,\ \R^2, \,\ \text{where} \,\ C:=\max\limits_{x\in \partial B_1(0)}h(x).
     \end{equation*}
Furthermore, if $h(x)\to\infty$ as $|x|\to\infty$, then the origin is the unique minimum point of $h(x)$.

Following the above definition, we next assume that $V_\Om(x):=V(x)-\f{\Om^2}{4}|x|^2$ satisfies
\begin{enumerate}
\item [($V_1$).] $V_{\Om}(x)\geq 0$, $\big\{x\in\R^2:\ V_\Om(x)=0\big\}=\{0\}$, and there exists a $\kappa>0$ such that $$V_\Om(x)+|\nabla V_\Om(x)|\leq Ce^{\kappa|x|} \,\ \text{as} \,\ |x|\to\infty.$$
\end{enumerate}

\begin{enumerate}
\item [($V_2$).]
There exists a homogeneous function $h(x)\in C^1(\R^2)$  of degree $1<s\leq2$, which satisfies  $\lim\limits_{|x|\to\infty}h(x)=+\infty$ and
$H(y):=\int_{\R^2}h(x+y)w^2(x)dx$ admits a unique global minimum point $y_0\in\R^2$, such that as $|x|\to0$,
\begin{equation}\label{A20-1}
V_\Om(x)=h(x)+o(|x|^s),\,\ \f{\partial V_\Om(x)}{\partial x_j}=\f{\partial h(x)}{\partial x_j}+o(|x|^{s-1}) \,\ , \,\ \hbox{where}\,\ j=1,2.
\end{equation}
\end{enumerate}

Under the assumptions \eqref{A5Vdecondition}, $(V_1)$ and $(V_2)$, we now give the following theorem on the limit behavior of minimizers for $I(\rho)$ as $\rho\to\infty$.
\vskip 0.2truein
\begin{thm}\label{thm.limibeha}
Suppose $V(x)$ satisfies (\ref{A5Vdecondition}), $(V_1)$ and $(V_2)$, and assume $0<\Omega<\Omega^*$, where $\Omega^*>0$ is defined by (\ref{A6Omdedefintion}). Denote $a^*:=\|w\|^2_2$, where $w(x)$ is the unique positive solution of (\ref{A4wdeequation}). Let $u_\rho$ be a minimizer of $I(\rho)$, then we have
     \begin{equation}\label{A12limibeha}
\lim_{\rho\to\infty}\ep_\rho u_\rho\big(\ep_\rho (x+y_0)\big)e^{-i\big(\f{\ep^2_\rho\Om}{2}x\cdot y^\perp_0-\theta_\rho\big)}=\f{w(x)}{\sqrt{a^*}}
     \end{equation}
strongly in $H^1(\R^2,\C)\cap L^\infty(\R^2,\C)$, where $y_0$ is defined in $(V_2)$, $\theta_\rho\in[0,2\pi)$ is a properly chosen constant, and $\ep_\rho$ is defined as
      \begin{equation}\label{A13epk}
     \ep_\rho:=\Big(\f{\rho}{\sqrt{a^*}}\Big)^{-\f{p-1}{3-p}}.
     \end{equation}
    \end{thm}
Theorem \ref{thm.limibeha} gives a detailed description of the limit behavior of minimizers $u_\rho$ for (\ref{A1cvp}). We shall encounter some new problems in the proof of Theorem \ref{thm.limibeha}. The first problem is that
one cannot use the Gagliardo-Nirenberg inequality directly to establish the optimal energy estimates in $L^2-$subcritical case here. In order to solve this problem, we need to introduce the following new constraint variational problem
\begin{equation}\label{A17Rcvp}
	\hat{I}(\rho):=\inf\limits_{u\in H^1(\R^2,\R),\|u\|^2_2=1}\hat{E}_\rho(u),
\end{equation}
where $\hat{E}_\rho(u)$ is defined by
\begin{equation}\label{A18Ref}
	\hat{E}_\rho(u):=\int_{\R^2}|\nabla u|^{2}dx-\frac{2\rho^{p-1}}{p+1}\int_{\R^2}|u|^{p+1}dx.
\end{equation}
We shall establish a refined energy estimate of $I(\rho)$ in Lemma \ref{lem.energyesti} by analyzing the energy estimate of $\hat{I}(\rho)$ and choosing some suitable test functions.
Another problem is that how to deal with the rotational term  $\Om\int_{\R^2}x^\perp\cdot(iu_\rho,\nabla u_\rho)dx$ in \eqref{A2ef} as $\rho\to\infty$, for which we need to make use of the refined energy estimate of $I(\rho)$ and some properties of the linearized operator \eqref{eqS3.43} in Section 3 below.

Finally, we shall analyze the local uniqueness of minimizers for $I(\rho)$ as $\rho\to\infty$. Assume that the unique global minimum point $y_0$ of $H(y)=\int_{\R^2}h(x+y)w^2(x)dx$ is non-degenerate in the sense that
\begin{equation}\label{nondege}
	\det\Big(\int_{\R^2}\f{\partial h(x+y_0)}{\partial x_j}\f{\partial w^2(x)}{\partial x_l}dx\Big)_{j,l=1,2}\neq0.
\end{equation}
then we have the following result concerning the uniqueness of minimizers.
\vskip 0.2truein
\begin{thm}\label{thm.localuniqueness}
Suppose $V(x)\in C^{1,\alpha}_{loc}(\R^2) (0<\alpha<1)$ satisfies (\ref{A5Vdecondition}), $(V_1)$ and $(V_2)$, and $\Om$ satisfies $0<\Om<\Om^*$, where $\Om^*>0$ is defined as in (\ref{A6Omdedefintion}). Moreover, we assume that
the unique global minimum point $y_0$ of $H(y)=\int_{\R^2}h(x+y)w^2(x)dx$ is  non-degenerate,
then up to the constant phase, there exists a unique complex-valued minimizer for $I(\rho)$ when $\rho>0$ is large enough.
    \end{thm}
We remark that the local uniqueness, up to a constant phase,  of Theorem \ref{thm.localuniqueness} holds in the following sense: there exists a minimizer  $U_\rho$ of $I(\rho)$ such that any minimizer $u_\rho$ of $I(\rho)$ satisfies $u_\rho\equiv U_\rho e^{i\theta_\rho}$ in $\R^2$ for $\rho>0$ large enough, where $\theta _\rho\in[0,2\pi)$ is a suitable constant phase depending on $\rho$ and $u_\rho$.

The main method of proving Theorem \ref{thm.localuniqueness} is inspired by \cite{CLL,DLY,Gi,GLW,GLP} and the references therein, but there still have some essential  difficulties occur in our proof.  Indeed, note that the Pohozaev identities play a crucial role on the process of the proof of local uniqueness \cite{CLL,DLY,Gi,GLW}. However, due to the appearance of the rotation term, the Euler-Lagrange equation \eqref{A11udeELE} of minimizers $u_\rho$ is essentially a coupled system of the real  and imaginary parts of $u_\rho$. Therefore, the first difficulty is that one cannot obtain the Pohozaev identities directly for the complex-valued minimizers of $I(\rho)$. To overcome this difficulty, we
shall construct various Pohozaev identities for the real part of the complex-valued minimizers and analyze  all terms produced by the
rotation.  The second difficulty is that those analysis in \cite{GLP}, which focus on the special case $p=4$, can not generalize to the proof of Theorem \ref{thm.localuniqueness} due to the rational of $1<p<3$. Therefore, we need to deal with the nonlinear term $|u_\rho|^{p-1}u_\rho$ more carefully. Besides, we shall use some technical expansions of the nonlinear term in \cite{CPY} flexibly  to obtain a desired estimates in our proof.


This paper is organized as follows. Section 2 is devoted to proving Theorem \ref{thm.exist} on the existence and nonexistence of minimizers for $I(\rho)$. In Section 3, we shall prove Theorem \ref{thm.limibeha} on the limit behavior of minimizers for $I(\rho)$ as $\rho\to\infty$ by employing energy methods and blow-up analysis. By deriving various Pohozaev identities, we shall complete the proof of the local uniqueness of minimizers in Section 4.

\section{Existence of minimizers for $I(\rho)$}
This section is concerned with the proof of Theorem \ref{thm.exist} on the existence and nonexistence of minimizers for $I(\rho)$. 
We first introduce the following compactness lemma.
    \begin{lem}\label{lem.com}
    Suppose that $V(x)\in L^{\infty}_{loc}(\R^2)$ and $\lim\limits_{|x|\to\infty}V(x)=\infty$. If $2\leq q<\infty$, then the embedding $\H\hookrightarrow L^q(\R^2,\C)$ is compact.
    \end{lem}
The proof of this lemma is similar to those in \cite{ZJ} and the references therein, so we omit it here.
   \qed
\vskip 0.2truein
\noindent{\bf{Proof of Theorem \ref{thm.exist}:}} \ Since the proof of Theorem \ref{thm.exist} is overall similar to those in \cite [Theorem 1.1] {GLY}, we only give the main idea of its proof here.

1.  For any $\rho\in(0,\infty)$ and $0<\Om<\Om^*$. Suppose that $u\in\H$ and $\|u\|^2_2=1$. Applying the Gagliardo-Nirenberg inequality (\ref{A7GN}) and the diamagnetic inequality (\ref{A10diamagnetic.inequ}), we deduce that that there exist sufficiently large $R>0$ and $C(\Om,\rho,R)>0$ such that for any $\rho\in(0,\infty)$ and $0<\Om<\Om^*$,
      \begin{equation}\label{11.19}
      E_\rho(u)\geq\f{1}{2}\int_{\R^2}|\nabla u|^2dx-C(\Om,\rho,R),
      \end{equation}
which implies that $E_\rho(u)$ is bounded from below. Let $\{u_n\}\in\H$ be a minimizing sequence of $I(\rho)$ satisfying $\|u_n\|^2_2=1$ and $\lim\limits_{n\to\infty}E_\rho(u_n)=I(\rho)$. It then follows from \eqref{11.19} that the sequence $\{u_n\}$ is bounded uniformly in $\H$. By the compact embedding in Lemma \ref{lem.com}, there exist a subsequence $\{u_{nk}\}$ of $\{u_n\}$ and $u_0\in\H$ such that
      \begin{equation*}
      u_{nk}\rightharpoonup u_0 \,\ \text{weakly in} \,\ \H, \,\ u_{nk}\to u_0 \,\ \text{strongly in} \ L^q(\R^2,\C)\,\ (2\leq q<\infty).
      \end{equation*}
Thus we conclude from above convergence and the weak lower semicontinuity that $\|u_0\|^2_2=1$ and $I(\rho)=E_\rho(u_0)$, $i.e.$, $u_0$ is a minimizer of $I(\rho)$. This implies that for any $\rho\in(0,\infty)$ and $0<\Om<\Om^*$, there exists at least one minimizer of $I(\rho)$.

2. For any $\rho\in(0,\infty)$ and $\Om>\Om^*$. Let $w=w(|x|)$ be the unique positive solution of (\ref{A4wdeequation}). For any $\tau>0$, choose a trail function
     \begin{equation}\label{eqS2.10}
     w_{\tau}(x):=\f{A_\tau\tau}{\|w\|_2}w\big(\tau (x-x_\tau)\big)\var(x-x_\tau)e^{i\Om S(x)},
     \end{equation}
where $A_{\tau}>0$ is chosen such that $\|w_\tau\|^2_2=1$, $x_\tau\in\R^2$ is chosen such that  $V_\Om(x_\tau)\leq-2\f{p-1}{p+1}\f{\tau^2\|w\|^{p+1}_{p+1}}{\|w\|^2_2}$, $S(x)=\f{1}{2}x\cdot x^\perp_\tau$, and $\var(x)\in C^\infty_0(\R^2)$ is a
cut-off function satisfying $\var(x)=1$ if $|x|\leq1$; $\var(x)=0$ if $|x|\geq2$; $\var(x)\in(0,1)$ if $1<|x|<2$.

Using the exponential decay of $w$ in (\ref{A8wdecay}) and the identity (\ref{A9identity}), direct calculations give that $E_\rho(w_\tau)\leq-\infty$ as $\tau\to\infty$, which implies that $I(\rho)$ is unbounded from below, and so $I(\rho)$ admits no minimizer for any $\rho\in(0,\infty)$ and $\Om>\Om^*$. This then completes the proof of Theorem \ref{thm.exist}.
\qed
\section{Mass concentration as $\rho\to\infty$}
This section is devoted to proving Theorem \ref{thm.limibeha} on the limit behavior of minimizers for $I(\rho)$ as $\rho\to\infty$. We shall first establish the energy estimates of $I(\rho)$ in Lemma \ref{lem.energyesti} and then present a detailed analysis on the limit behavior of minimizers for $I(\rho)$ as $\rho\to\infty$. Based on these energy estimates and analysis, we finally complete the proof of Theorem \ref{thm.limibeha}.
\subsection{Energy estimates of $I(\rho)$}
We recall the following energy estimates of $\hat{I}(\rho)$ defined in (\ref{A17Rcvp}) as $\rho\to\infty$.
    \begin{lem}[{\cite [Lemma A.3]{LZ}}]\label{lemma A.3}
Let $\hat{u}_\rho$ be a nonnegative minimizer of $\hat{I}(\rho)$ defined in (\ref{A17Rcvp}). Set $a^*:=\|w\|^2_2$, where $w$ is the unique positive solution of (\ref{A4wdeequation}). Then we have
     \begin{equation}\label{ARefesti}
     \hat{I}(\rho)=-\f{3-p}{2}\ep^{-2}_\rho,
     \end{equation}
and there exists a $x_0\in\R^2$ such that
     \begin{equation}\label{ARmini}
     \hat{u}_\rho(x)=\f{1}{\sqrt{a^*}}\ep^{-1}_\rho w(\ep^{-1}_\rho x+x_0),
     \end{equation}
where $\ep_\rho$ is defined in (\ref{A13epk}).
     \end{lem}
Based on Lemma \ref{lemma A.3}, we have the following energy estimates of $I(\rho)$.
     \begin{lem}\label{lem.energyesti}
Suppose that $V(x)$ satisfies (\ref{A5Vdecondition}), $(V_1)$ and assume $0<\Om<\Om^*$, where $\Om^*>0$ is defined in (\ref{A6Omdedefintion}). Let $u_\rho$ be a minimizer of $I(\rho)$ for any $0<\rho<\infty$. Then we have
     \begin{equation}\label{energyesti}
     0\leq I(\rho)-\hat I(\rho)\to0\,\ \text{as} \,\ \rho\to\infty,
     \end{equation}
and
     \begin{equation}\label{Vto0}
     \int_{\R^2}V_\Om(x)|u_\rho|^2dx\to0 \,\ \text{as} \,\ \rho\to\infty.
     \end{equation}
     \end{lem}
     \begin{proof}
By (\ref{A2ef}), the diamagnetic inequality (\ref{A10diamagnetic.inequ}), (\ref{A17Rcvp}) and $(V_1)$, we have
     \begin{equation}\label{lowerboundesti}
      \begin{split}
      I(\rho)
      &=\int_{\R^2}|(\nabla-i\A)u_\rho|^2dx+\int_{\R^2}V_\Om(x)|u_\rho|^2dx-\f{2\rho^{p-1}}{p+1}\int_{\R^2}|u_\rho|^{p+1}dx\\
      &\geq \int_{\R^2}\big|\nabla|u_\rho|\big|^2dx-\f{2\rho^{p-1}}{p+1}\int_{\R^2}|u_\rho|^{p+1}dx\\
      &\geq \hat{I}(\rho)\quad \text{as} \ \rho\to\infty.
     \end{split}
     \end{equation}
Taking a test function $w_\tau$ of the form (\ref{eqS2.10}) with $S(x)\equiv 0$, $x_\tau=0$ and $\tau=\Big(\f{\rho}{\sqrt{a^*}}\Big)^{\f{p-1}{3-p}}$. Applying the exponential decay of $w$ in (\ref{A8wdecay}) and the identity (\ref{A9identity}), direct calculations give
     \begin{equation}\label{supperboundesti}
      \begin{split}
      I(\rho)\leq E_\rho(w_\tau)
      &\leq \f{p-1}{2}\tau^2-\Big(\f{\rho}{\sqrt{a^*}}\Big)^{p-1}\tau^{p-1}+C\tau^{-2}\\
      &=\hat I(\rho)+o(1) \,\ \text{as} \,\ \rho\to\infty.
       \end{split}
     \end{equation}
Then (\ref{energyesti}) follows from (\ref{lowerboundesti}) and (\ref{supperboundesti}).

Now, we shall prove (\ref{Vto0}). Combining (\ref{A2ef}), the diamagnetic inequality (\ref{A10diamagnetic.inequ}), (\ref{A17Rcvp}) and (\ref{energyesti}), one then deduce that
     \begin{equation}\label{1eqS3.8}
     \begin{split}
     0\leq\int_{\R^2}V_\Om(x)|u_\rho|^2dx
     &=I(\rho)-\int_{\R^2}|(\nabla-i\A)u_\rho|^2dx+\f{2\rho^{p-1}}{p+1}\int_{\R^2}|u_\rho|^{p+1}dx\\
     &\leq I(\rho)-\Big(\int_{\R^2}\big|\nabla|u_\rho|\big|^2dx-\f{2\rho^{p-1}}{p+1}\int_{\R^2}|u_\rho|^{p+1}dx\Big)\\
     &=I(\rho)-\hat{E}_\rho(|u_\rho|)\\
     &\leq I(\rho)-\hat{I}(\rho)\to0 \,\ \text{as} \,\ \rho\to\infty,
     \end{split}
     \end{equation}
which implies (\ref{Vto0}). This completes the proof of Lemma \ref{lem.energyesti}.
\end{proof}
\subsection{Blowup analysis}
The main purpose of this subsection is to present a detailed analysis on the limit behavior of minimizers for $I(\rho)$ as $\rho\to\infty$. Towards this aim, we first give the following lemma on the refined estimates of minimizers $u_\rho$ and its Lagrange multiplier $\mu_\rho$ as $\rho\to\infty$.
     \begin{lem}\label{lem.important}
Suppose that $V(x)$ satisfies (\ref{A5Vdecondition}), $(V_1)$ and assume $0<\Om<\Om^*$, where $\Om^*>0$ is defined in (\ref{A6Omdedefintion}). Let $u_\rho$ be a minimizer of $I(\rho)$, then we have
\begin{enumerate}
\item Define
     \begin{equation}\label{eqS3.8}
      w_\rho(x):=\ep_\rho u_\rho(\ep_\rho x+z_\rho)e^{-i\big(\f{\ep_\rho\Om}{2}x\cdot z^\perp_\rho-\theta_\rho\big)},
     \end{equation}
where $\ep_\rho$ is defined by (\ref{A13epk}), $z_\rho$ is a global maximum point of $|u_\rho|$ and $\theta_\rho\in[0,2\pi)$ is a proper constant. Then there exists a constant $\eta>0$, independent of $0<\rho<\infty$, such that
     \begin{equation}\label{eqS3.8-1}
     \int_{B_2(0)}|w_\rho(x)|^2dx\geq\eta>0 \,\ \text{as} \,\ \rho\to\infty.
     \end{equation}
\item $w_\rho$ satisfies
     \begin{equation}\label{eqS3.9}
     \lim\limits_{\rho\to\infty}w_\rho(x)=\f{w(x)}{\sqrt{a^*}} \,\ \text{strongly in} \,\ H^{1}(\R^2,\C),
     \end{equation}
where $a^{\ast}:=\|w\|^{2}_{2}$ and $w$ is the unique positive solution of (\ref{A4wdeequation}). Furthermore, any global maximal point $z_\rho$ of $|u_\rho|$ satisfies $\lim\limits_{\rho\to\infty}V_\Om(z_\rho)=0$.
\item $\mu_\rho\epsilon^2_\rho\to-1$ as $\rho\to\infty$.
\end{enumerate}
     \end{lem}
     \begin{proof}
     1. Denote $\bar{w}_\rho(x):=\ep_\rho u_\rho(\ep_\rho x+z_\rho)e^{-i\f{\ep_\rho\Om}{2}x\cdot z^\perp_\rho}$ and $w_\rho(x):=\bar{w}_\rho(x) e^{i\theta_\rho}$, where the parameter $\theta_\rho\in[0,2\pi)$ is chosen properly such that
     \begin{equation}\label{eqS3.11}
     \|w_\rho-\f{w}{\sqrt{a^*}}\|_{L^2(\R^2)}
     =\min\limits_{\theta\in[0,2\pi)}\|e^{i\theta}\bar{w}_\rho-\f{w}{\sqrt{a^*}}\|_{L^2(\R^2)}.
     \end{equation}
Rewrite
     \begin{equation}\label{eqS3.12}
     w_\rho(x)=R_\rho(x)+iI_\rho(x),
     \end{equation}
where $R_\rho(x)$ and $I_\rho(x)$ denote the real and imaginary parts of $w_\rho(x)$  respectively. By (\ref{eqS3.11}), we have
     \begin{equation}\label{eqS3.13}
     \int_{\R^2}w(x)I_\rho(x)dx=0.
     \end{equation}
From (\ref{A11udeELE}) and (\ref{eqS3.8}), we deduce that $w_\rho$ satisfies
     \begin{equation}\label{eqS3.26}
     \begin{split}
     &-\Delta w_\rho+i\ep^2_\rho\Om(x^\perp\cdot\nabla w_\rho)\\
     &+\Big[\f{\ep^4_\rho\Om^2}{4}|x|^2
     +\ep^2_\rho V_\Om(\ep_\rho x+z_\rho)-\ep^2_\rho\mu_\rho-(a^*)^\f{p-1}{2}|w_\rho|^{p-1}\Big]w_\rho=0\,\ \text{in} \,\ \R^2.
     \end{split}
     \end{equation}
Set $W_\rho=|w_\rho|^2\geq0$. We then obtain from (\ref{eqS3.26}) that
     \begin{equation}\label{eqS3.27}
     \begin{split}
     &-\f{1}{2}\Delta W_\rho+|\nabla w_\rho|^2-\ep^2_\rho\Om x^{\perp}\cdot(iw_\rho,\nabla w_\rho)\\
     &+\Big[\f{\ep^4_\rho\Om^2}{4}|x|^2
     +\ep^2_\rho V_\Om(\ep_\rho x+z_\rho)-\ep^2_\rho\mu_\rho-(a^*)^{\f{p-1}{2}}W^\f{p-1}{2}_\rho\Big]W_\rho=0 \,\ \text{in} \,\ \R^2.
     \end{split}
     \end{equation}
Using the diamagnetic inequality (\ref{A10diamagnetic.inequ}), we derive from (\ref{eqS3.27}) that
     \begin{equation}\label{eqS3.28}
     -\f{1}{2}\Delta W_\rho+[-\ep^2_\rho\mu_\rho-(a^*)^{\f{p-1}{2}}W^\f{p-1}{2}_\rho]W_\rho\leq0 \,\ \text{in} \,\ \R^2.
     \end{equation}
By \eqref{1.12}, \eqref{ARefesti} and \eqref{energyesti}, we obtain that $\ep^2_\rho\mu_\rho\leq-\f{3-p}{2}$ as $\rho\to\infty$. Note that $W^\f{p-1}{2}_\rho$ is bounded uniformly in $L^{\f{2}{p-1}}(\R^2,\R)$, where $1<\f{2}{p-1}<\infty$. Then applying De Giorgi-Nash-Moser theory \cite [Theorem 4.1]{HL} to (\ref{eqS3.28}) yields that
     \begin{equation}\label{K1}
     \int_{B_2(0)}W_\rho(x) dx\geq\max_{x\in B_1(0)}W_\rho(x) \,\ \text{as} \,\ \rho\to\infty.
     \end{equation}
Since 0 is a global maximal point of $W_\rho(x)$ for any $0<\rho<\infty$, we have $-\Delta W_\rho(0)\geq0$ for all $0<\rho<\infty$. Using the fact that $\ep^2_\rho\mu_\rho\leq-\f{3-p}{2}$ as $\rho\to\infty$, we then obtain from \eqref{eqS3.28} that there exists a constant $\beta>0$, which is independent of $0<\rho<\infty$, such that
     \begin{equation}\label{K1-1}
     W_\rho(0)\geq\beta>0.
     \end{equation}
It then follows from \eqref{K1} and \eqref{K1-1} that \eqref{eqS3.8-1} holds.

2. Using the diamagnetic inequality, we derive from (\ref{A10diamagnetic.inequ}), (\ref{ARefesti}), (\ref{energyesti}) and (\ref{eqS3.8})  that
     \begin{align}\label{eqS3.15}
      -\f{3-p}{2}
      &=\ep^2_\rho\hat I(\rho)=\ep^2_\rho I(\rho)+o(1)\nonumber\\
      &=\int_{\R^2}\Big(|\nabla w_\rho|^2
      -\ep^2_\rho\Om x^\perp\cdot(iw_\rho,\nabla w_\rho)
      +\f{\Om^2}{4}\ep^4_\rho|x|^2|w_\rho|^2+\ep^2_\rho V_\Om(\ep_\rho x+z_\rho)|w_\rho|^2\nonumber\\
      &\quad-\f{2(a^{\ast})^{\f{p-1}{2}}}{p+1}|w_\rho|^{p+1}\Big)dx
      +o(1)\nonumber\\
      &\geq\int_{\R^2}\big|\nabla|w_\rho|\big|^2dx
      -\f{2(a^{\ast})^{\f{p-1}{2}}}{p+1}\int_{\R^2}|w_\rho|^{p+1}dx\nonumber\\
      &\geq\hat{I}(\sqrt{a^*})=-\f{3-p}{2}\,\ \text{as} \,\ \rho\to\infty,
     \end{align}
which yields that $|w_\rho|$ is a minimizing sequence of $\hat{I}(\sqrt{a^*})$. Note from \eqref{eqS3.15} that $|w_\rho|$ is bounded uniformly in $H^1(\R^2,\R)$, therefore we can assume that up to a subsequence if necessary, $|w_\rho|$ convergence to $w_0$ weakly in $H^1(\R^2,\R)$ as $\rho\to\infty$ for some $0\leq w_0\in H^1(\R^2,\R)$. From \eqref{eqS3.8-1}, we get that $w_0\not\equiv0$ in $\R^2$. By the weak convergence, we may assume that $|w_\rho|\to w_0$ $a.e.$ in $\R^2$ as $\rho\to\infty$. Using the Br\'{e}zis-Lieb  lemma, we obtain that
    \begin{equation}\label{3.15-1}
     \|w_\rho\|_q^q=\|w_0\|_q^q+\big\||w_\rho|-w_0\big\|_q^q+o(1)\,\ \mbox{as} \,\ \rho\to\infty, \,\ \hbox{where} \,\ 2\leq q< \infty,
    \end{equation}
and
    \begin{equation}\label{3.15-2}
    \big\|\nabla |w_\rho|\big\|_2^2=\big\|\nabla  w_0 \big\|_2^2+\big\|\nabla (|w_\rho|-w_0)\big\|_2^2+o(1) \,\ \mbox{as} \,\ \rho\to\infty.
    \end{equation}
Next, we prove that $\|w_0\|^2_2=1$. On the contrary, we assume that $\|w_0\|^2_2=\lambda$ and $\big\||w_\rho|-w_0\big\|_2^2=1-\lambda$, where $\lambda\in(0,1)$. Set $w_\lambda:=\f{w_0}{\sqrt{\lambda}}$ and $w_{1-\lambda}:=\f{|w_\rho|-w_0}{\sqrt{1-\lambda}}$. From \eqref{ARefesti}, \eqref{eqS3.15}, \eqref{3.15-1} and \eqref{3.15-2}, we then derive that as $\rho\to\infty$,
     \begin{align}\label{eqS3.15-3}
      -\f{3-p}{2}
      &=\lim_{\rho\to\infty}\Big[\int_{\R^2}\big|\nabla|w_\rho|\big|^2dx
      -\f{2(a^{\ast})^{\f{p-1}{2}}}{p+1}\int_{\R^2}|w_\rho|^{p+1}dx\Big]\notag\\
      &=\int_{\R^2}\big|\nabla w_0\big|^2dx
      -\f{2(a^{\ast})^{\f{p-1}{2}}}{p+1}\int_{\R^2}|w_0|^{p+1}dx\notag\\
      &\quad+\lim_{\rho\to\infty}\Big[\int_{\R^2}\big|\nabla(|w_\rho|-w_0)\big|^2dx
      -\f{2(a^{\ast})^{\f{p-1}{2}}}{p+1}\int_{\R^2}\big||w_\rho|-w_0\big|^{p+1}dx\Big]\notag\\
      &>\lambda\Big[\int_{\R^2}\big|\nabla w_\lambda\big|^2dx
      -\f{2(a^{\ast})^{\f{p-1}{2}}}{p+1}\int_{\R^2}|w_\lambda|^{p+1}dx\Big]\notag\\
      &\quad+(1-\lambda)\lim_{\rho\to\infty}\Big[\int_{\R^2}\big|\nabla w_{1-\lambda}\big|^2dx-\f{2\big(a^{\ast}\big)^{\f{p-1}{2}}}{p+1}
      \int_{\R^2}|w_{1-\lambda}|^{p+1}dx\Big]\notag\\
      &\geq\lambda\hat{I}(\sqrt{a^*})+(1-\lambda)\hat{I}(\sqrt{a^*})
      =\hat{I}(\sqrt{a^*})=-\f{3-p}{2},\notag
     \end{align}
which is a contradiction, therefore $\|w_0\|^2_2=1$ holds. Since $\|w_\rho\|^2_2=\|w_0\|^2_2=1$, we obtain from \eqref{3.15-1} that
     \begin{equation}\label{eqS3.15-4}
     |w_\rho(x)|\to w_0(x) \,\ \text{strongly in} \,\ L^2(\R^2,\R)  \,\ \text{as} \,\ \rho\to\infty.
     \end{equation}
By the uniform boundedness of $|w_\rho|$ in $H^1(\R^2,\R)$ and the interpolation inequality, we derive from \eqref{eqS3.15-4} that $|w_\rho(x)|\to w_0(x)$ strongly in $L^q(\R^2,\R)$ $(2\leq q<\infty)$ as $\rho\to\infty$. Moreover, using the weak lower semicontinuity, \eqref{ARefesti} and \eqref{eqS3.15}, we obtain that $\nabla|w_\rho(x)|\to \nabla w_0(x)$ strongly in $L^2(\R^2,\R)$ as $\rho\to\infty$. Therefore, we deduce from above that
     \begin{equation}\label{eqS3.16}
     |w_\rho(x)|\to w_0(x) \,\ \text{strongly in} \,\ H^1(\R^2,\R) \,\ \text{as} \,\ \rho\to\infty.
     \end{equation}
Since $|w_\rho|$ is a minimizing sequence of $\hat{I}(\sqrt{a^*})$ and $\|w_0\|^2_2=1$, we obtain from \eqref{eqS3.16} that $w_0$ is a minimizer of $\hat{I}(\sqrt{a^*})$. Then we get from \eqref{ARmini} that $w_0(x)=\f{w(x+x_0)}{\sqrt{a^*}}$, where $x_0\in\R^2$. On the other hand, since the origin is a global maximum point of $|w_\rho|$ for any $\rho\in(0,\infty)$, it must be also a global maximum point of $w(x+x_0)$ in view of (\ref{eqS3.16}), which implies that $x_0=0$. We conclude that, up to a subsequence if necessary,
     \begin{equation}\label{eqS3.17}
     |w_\rho(x)|\to \f{w(x)}{\sqrt{a^*}} \,\ \text{strongly in} \,\ H^1(\R^2,\R) \,\ \text{as} \,\ \rho\to\infty.
     \end{equation}
Furthermore, since the convergence (\ref{eqS3.17}) is independent of what subsequence $\{|w_\rho|\}$ we choose, we conclude that (\ref{eqS3.17}) holds for the whole sequence.

Now, we shall prove that $w_\rho$ is bounded uniformly in $H^1(\R^2,\C)$ as $\rho\to\infty$. By the definition of $w_\rho$ in (\ref{eqS3.8}), we only need to prove that there exists a constant $C>0$, independent of $\rho$, such that as $\rho\to\infty$,
     \begin{equation}\label{eqS3.18}
    \int_{\R^2}|\nabla w_\rho|^2dx\leq C.
     \end{equation}
\noindent In fact, since $0<\Om<\Om^*$ is fixed, the definition (\ref{A6Omdedefintion}) of $\Om^*$ implies that
    \begin{equation}\label{eqS3.20}
    |x|^2\leq C(\Om)\Big(V(x)-\f{\Om^2}{4}|x|^2\Big) \,\ \text{for sufficiently large} \,\ |x|>0.
    \end{equation}
Using (\ref{eqS3.20}), we then deduce that for any given large constant $M>0$,
     \begin{align}\label{eqS3.21}
     \f{\Om^2}{4}\ep^4_\rho\int_{\R^2}|x|^2|w_\rho|^2dx
     &=\f{\Om^2}{4}\ep^2_\rho\int_{|\ep_\rho x|\leq M}|\ep_\rho x|^2|w_\rho|^2dx
     +\f{\Om^2}{4}\ep^2_\rho\int_{|\ep_\rho x|> M}|\ep_\rho x|^2|w_\rho|^2dx\notag\\
     &\leq o(1)+C(\Om)\int_{|\ep_\rho x|>M}\ep^2_\rho V_\Om(\ep_\rho x)|w_\rho|^2dx\\
     &=o(1)  \,\ \text{as} \,\ \rho\to\infty.\notag
     \end{align}
By H\"{o}lder inequality, it then follows from (\ref{eqS3.21}) that as $\rho\to\infty$,
     \begin{equation}\label{eqS3.22}
     \begin{split}
     \ep^2_\rho\Om\Big|\int_{\R^2}x^\perp\cdot(iw_\rho,\nabla w_\rho)dx\Big|
     &\leq\Big(\ep^4_\rho\Om^2\int_{\R^2}|x|^2|w_\rho|^2dx\Big)^{\f{1}{2}}
     \Big(\int_{\R^2}|\nabla w_\rho|^2dx\Big)^{\f{1}{2}}\\
     &=o(1)\Big(\int_{\R^2}|\nabla w_\rho|^2dx\Big)^{\f{1}{2}}.
     \end{split}
     \end{equation}
Since $|w_\rho|$ is bounded  uniformly  in $H^1(\R^2,\R)$, we derive from (\ref{eqS3.22}) and (\ref{eqS3.15}) that as $\rho\to\infty$,
     \begin{equation}\label{eqS3.23}
     -\f{3-p}{2}
     \geq\int_{\R^2}|\nabla w_\rho|^2dx-o(1)\Big(\int_{\R^2}|\nabla w_\rho|^2dx\Big)^{\f{1}{2}}-C,
     \end{equation}
which implies (\ref{eqS3.18}) holds true.

Based on (\ref{eqS3.17}) and \eqref{eqS3.18}, we may assume that up to a subsequence if necessary,
     \begin{equation}\label{g1}
     w_\rho\rightharpoonup\hat w_0 \,\ \text{weakly in} \,\ H^1(\R^2,\C) \,\ \text{as} \,\ \rho\to\infty,
     \end{equation}
and
     \begin{equation}\label{g2}
     w_\rho\to \hat w_0 \,\ \text{strongly in} \,\ L^q_{loc}(\R^2,\C) (2\leq q<\infty) \,\ \text{as} \,\ \rho\to\infty,
     \end{equation}
for some $\hat w_0\in H^1(\R^2,\C)$ and $\hat w_0\not\equiv0$. From (\ref{eqS3.17}) and (\ref{g2}), we derive that
     \begin{equation}\label{g3}
     \begin{split}
     \int_{\R^2}|\hat w_0|^2dx&=\lim\limits_{R\to\infty}\lim\limits_{\rho\to\infty}
      \int_{B_R(0)}|w_\rho|^2dx\\
     &=\lim\limits_{R\to\infty}
      \int_{B_R(0)}\big|\f{w}{\sqrt{a^*}}\big|^2dx=1,
      \end{split}
     \end{equation}
which together with (\ref{g1}) implies that
     \begin{equation}\label{g4}
     w_\rho\to \hat w_0 \quad \text{strongly in} \ L^2(\R^2,\C) \ \text{as} \ \rho\to\infty.
     \end{equation}
Since $w_\rho$ is bounded uniformly in $H^1(\R^2,\C)$ as $\rho\to\infty$, using the interpolation inequality, we derive from (\ref{g4}) that
     \begin{equation}\label{g5}
     w_\rho\to \hat w_0 \quad \text{strongly in} \ L^q(\R^2,\C) (2\leq q<\infty) \ \text{as} \ \rho\to\infty.
     \end{equation}
By the weak lower semicontinuity, (\ref{A17Rcvp}), (\ref{ARefesti}), \eqref{eqS3.22} and (\ref{g5}), we deduce from (\ref{eqS3.15}) that
      \begin{equation}\label{g6}
      \begin{split}
      \hat I(\sqrt{a^*})=-\f{3-p}{2}
      &\geq\lim\limits_{\rho\to\infty}\int_{\R^2}\Big(|\nabla w_\rho|^2
      -\f{2(a^{\ast})^{\f{p-1}{2}}}{p+1}|w_\rho|^{p+1}\Big)dx\\
      &\geq\int_{\R^2}\Big(|\nabla \hat w_0|^2
      -\f{2(a^{\ast})^{\f{p-1}{2}}}{p+1}|\hat w_0|^{p+1}\Big)dx\\
      &\geq\hat I(\sqrt{a^*}),
      \end{split}
     \end{equation}
which then yields
      \begin{equation}\label{g7}
      \lim\limits_{\rho\to\infty}\int_{\R^2}|\nabla w_\rho|^2
      =\int_{\R^2}|\nabla \hat w_0|^2dx.
     \end{equation}
Using (\ref{eqS3.17}) and \eqref{g7}, we obtain that
      \begin{equation}\label{g8}
      \int_{\R^2}|\nabla \hat w_0|^2
      =\int_{\R^2}|\nabla |\hat w_0||^2dx, \ i.e., \ |\nabla \hat w_0|
      =|\nabla |\hat w_0|| \,\  a.e. \,\ \text{ in} \,\ \R^2.
     \end{equation}
We thus conclude from (\ref{eqS3.17}), (\ref{g5}), (\ref{g7}) and (\ref{g8}) that
     \begin{equation}\label{g9}
     w_\rho\to\hat w_0=\f{w}{\sqrt{a^*}}e^{i\sigma} \,\ \text{strongly in} \,\ H^1(\R^2,\C) \,\ \text{as} \,\ \rho\to\infty
     \end{equation}
for some $\sigma\in\R$. Furthermore, it follows from (\ref{eqS3.13}) that $\sigma=0$. Since the convergence of (\ref{g9}) is independent of the choice of the subsequence, we conclude that (\ref{g9}) holds for the whole sequence and hence (\ref{eqS3.9}) holds. From (\ref{Vto0}) and (\ref{eqS3.9}), we obtain that $\lim\limits_{\rho\to\infty}V_\Om(z_\rho)=0$.

3.  From (\ref{A9identity}), (\ref{1.12}), (\ref{eqS3.8}) and (\ref{g5}), we deduce that as $\rho\to\infty$,
     \begin{equation}\label{eqS3.25}
      \begin{split}
      \mu_\rho\ep^2_\rho&=\ep^2_\rho\big[I(\rho)-\f{p-1}{p+1}\rho^{p-1}\int_{\R^2}|u_\rho|^{p+1}dx\big]\\
                  &=\ep^2_\rho I(\rho)-\f{p-1}{p+1}(a^*)^\f{p-1}{2}\int_{\R^2}|w_\rho|^{p+1}dx\\
                  &\to
                  -1.
     \end{split}
     \end{equation}
This therefore completes the proof of Lemma \ref{lem.important}.
\end{proof}
    \begin{lem}\label{decay-L}
Under the assumptions of Theorem \ref{thm.limibeha} and let $u_\rho$ be a minimizer of $I(\rho)$ and $w_\rho$ be defined in Lemma \ref{lem.important}. Then we have
\begin{enumerate}
\item $w_\rho$ decays exponentially in the sense that
     \begin{equation}\label{eqS3.10}
      |w_\rho(x)|\leq Ce^{-\f{2}{3}|x|} \,\ \text{in} \,\ \R^2/B_R(0) \,\ \text{as} \,\ \rho\to\infty,
     \end{equation}
where $C>0$ is a constant independent of $\rho$ and $R>0$.
\item The global maximum point $z_\rho$ of $|u_\rho|$ is unique as $\rho\to\infty$, and $w_\rho(x)$ satisfies
     \begin{equation}\label{A12limibeha-1}
      w_\rho(x)\to \f{w(x)}{\sqrt{a^*}} \,\ \text{uniformly in} \,\ L^\infty(\R^2,\C) \,\ \text{as} \,\ \rho\to\infty.
     \end{equation}
\end{enumerate}
    \end{lem}
Since the proof of this lemma is similar to those in \cite [Proposition 3.3]{GLY}, we omit it here.\qed
\vskip 0.2truein
\noindent{\bf{Proof of Theorem \ref{thm.limibeha}:}} \ In view of Lemmas \ref{lem.important} and \ref{decay-L}, in order to complete the proof of Theorem \ref{thm.limibeha}, it remains to prove that
     \begin{equation}\label{A14concentrate.rate}
     \lim\limits_{\rho\to\infty}\f{z_\rho}{\ep_\rho}=y_0,
     \end{equation}
where $z_\rho$ is the unique global maximum point of $|u_\rho|$, $\ep_\rho$ and $y_0$ are defined by (\ref{A13epk}) and $(V_2)$ respectively.

Setting $\tilde{u}_\rho(x):=\f{1}{\sqrt{a^*}}\ep^{-1}_\rho w(\ep^{-1}_\rho x-y_0)$. Recall from (\ref{ARmini}) that $\tilde{u}_\rho$ is a nonnegative minimizer of $\hat{I}(\rho)$, it then follows from (\ref{A8wdecay}), $(V_1)$ that as $\rho\to\infty$,
     \begin{align}\label{eqS3.36}
      &\quad I(\rho)-\hat{I}(\rho)\nonumber\\
      &\leq E_{\rho}\big(\tilde{u}_\rho(x)e^{i\f{\ep_\rho\Om}{2}x\cdot y^\perp_0}\big)
      -\hat{E}_{\rho}\big(\tilde{u}_\rho(x)\big)\nonumber\\
      &=\f{\ep^2_\rho}{a^*}\f{\Om^2}{4}\int_{\R^2}|x|^2w^2(x)dx
      +\f{1}{a^*}\int_{\R^2}V_{\Om}(\ep_\rho x+\ep_\rho y_0)w^2(x)dx\nonumber\\
   &=
   \left\{
   \begin{array}{lr}
   \f{\ep^s_\rho}{a^*}\big(1+o(1)\big)\ds\int_{\R^2}h(x+y_0)w^2(x)dx &\text{if} \,\ 1<s<2,\\
   \f{\ep^2_\rho}{a^*}\big(1+o(1)\big)\ds\int_{\R^2}\Big(\f{\Om^2}{4}|x|^2w^2(x)+h(x+y_0)w^2(x)\Big)dx &\text{if} \,\ s=2.
   \end{array}
   \right.
     \end{align}
 On the other hand, we derive from (\ref{A1cvp}), (\ref{A17Rcvp}) and (\ref{eqS3.8}) that as $\rho\to\infty$,
     \begin{equation}\label{eqS3.37}
     \begin{split}
      &\quad I(\rho)-\hat{I}(\rho)\\
      &\geq E_{\rho}(u_\rho)-\hat{E}_{\rho}(|u_\rho|)\\
      &\geq\int_{\R^2}V_\Om(\ep_\rho x\!+\!z_\rho)|w_\rho|^2dx
      \!+\!\f{\Om^2}{4}\ep^2_\rho\int_{\R^2}|x|^2|w_\rho|^2dx
      \!-\!\Om \int_{\R^2}x^\perp\cdot(iw_\rho,\nabla w_\rho)dx.
     \end{split}
     \end{equation}
     Next, we claim that
     \begin{equation}\label{eqS3.38}
       \Om\int_{\R^2}x^\perp\cdot(iw_\rho,\nabla w_\rho)dx=o(\epsilon^{1+\f{s}{2}}_\rho) \,\ \text{as} \,\ \rho\to\infty.
     \end{equation}
     Actually, one can derive from (\ref{A1cvp}), \eqref{ARefesti}, (\ref{eqS3.8}) and (\ref{eqS3.36}) that
      \begin{align}\label{eqS3.39}
      C\epsilon^{2+s}_\rho&\geq \ep^2_\rho I(\rho)-\ep^2_\rho\hat I(\rho)\nonumber\\
      &\geq\int_{\R^2}\Big(|\nabla w_\rho|^2dx
      -\f{2(a^*)^\f{p-1}{2}}{p+1}\int_{\R^2}|w_\rho|^{p+1}\Big)dx
      -\ep^2_\rho\Om\int_{\R^2}x^\perp\cdot(iw_\rho,\nabla w_\rho)dx\nonumber\\
      &\quad+\f{3-p}{2}\quad \text{as}\,\ \rho\to\infty.
     \end{align}
Note from (\ref{eqS3.12}), without loss of generality,
we assume that $\|R_\rho\|^2_{L^2}=\lambda$, $\|I_\rho\|^2_{L^2}=1-\lambda$, where $\lambda\in(0,1]$. In addition, it follows from (\ref{A12limibeha-1}) that
$R_\rho\to\f{w}{\sqrt{a^*}}$ and $I_\rho\to0$ uniformly in $\R^2$ as $\rho\to\infty$. Then, we derive from
 (\ref{A17Rcvp}), (\ref{ARefesti}), (\ref{ARmini}) and above that as $\rho\to\infty$,
 \begin{align}\label{eqS3.40}
      &\quad\int_{\R^2}|\nabla w_\rho|^2dx-\f{2(a^*)^\f{p-1}{2}}{p+1}\int_{\R^2}|w_\rho|^{p+1}dx\nonumber\\
     &=\int_{\R^2}\Big[|\nabla R_\rho|^2+|\nabla I_\rho|^2\Big]dx
     -\f{2(a^*)^\f{p-1}{2}}{p+1}\int_{\R^2}
     \Big[R^{p+1}_\rho+\f{p+1}{2}R^{p-1}_\rho I^2_\rho+o(I^2_\rho)\Big]dx\nonumber\\
     &=\lambda\Big(\int_{\R^2}\Big|\nabla \f{R_\rho}{\sqrt{\lambda}}\Big|^2dx-\f{2(a^*)^\f{p-1}{2}}{p+1}\int_{\R^2}\Big|\f{R_\rho}{\sqrt{\lambda}}\Big|^{p+1}dx\Big)
     +\f{2(a^*)^\f{p-1}{2}}{p+1}(\lambda-\lambda^{\f{p+1}{2}})\int_{\R^2}\Big|\f{R_\rho}{\sqrt{\lambda}}\Big|^{p+1}dx\nonumber\\
     &\quad+\int_{\R^2}|\nabla I_\rho|^2dx-(a^*)^{\f{p-1}{2}}\int_{\R^2}R^{p-1}_\rho I^2_\rho dx+o\big(\|I_\rho\|^2_2\big)\nonumber\\
     &\geq\lambda \hat{I}(\sqrt{a^*})+\f{2(a^*)^\f{p-1}{2}}{p+1}\f{p-1}{2}(1-\lambda)\int_{\R^2}\Big|\f{R_\rho}{\sqrt{\lambda}}\Big|^{p+1}dx\nonumber\\
     &\quad+\int_{\R^2}|\nabla I_\rho|^2dx-(a^*)^{\f{p-1}{2}}\int_{\R^2}R^{p-1}_\rho I^2_\rho dx
     +o\big(\|I_\rho\|^2_2\big)\nonumber\\
     &=
     -\f{3-p}{2}+\int_{\R^2}I^2_\rho dx
     +\int_{\R^2}|\nabla I_\rho|^2dx-\int_{\R^2}w^{p-1}I^2_\rho dx
     +o\big(\|I_\rho\|^2_2\big),
     \end{align}
where $\|\f{R_\rho}{\sqrt{\lambda}}\|^2_2=1$ is used in the $``\geq"$.
On the other hand, it follows from (\ref{eqS3.10}) and (\ref{A12limibeha-1}) that
      \begin{equation}\label{eqS3.41}
      \begin{split}
      \ep^2_\rho\Om\int_{\R^2}x^\perp\cdot(iw_\rho,\nabla w_\rho)dx
      &=\ep^2_\rho\Om\int_{\R^2}x^\perp\cdot(R_\rho\nabla I_\rho-I_\rho\nabla R_\rho)dx\\
      &=2\ep^2_\rho\Om\int_{\R^2}x^\perp\cdot(R_\rho\nabla I_\rho)dx\leq C\ep^2_\rho\|\nabla I_\rho\|_{L^2}.
     \end{split}
     \end{equation}
Combining (\ref{eqS3.39})-(\ref{eqS3.41}), we obtain that
      \begin{align}\label{eqS3.42}
      C\epsilon^{2+s}_\rho
      &\geq\int_{\R^2}I^2_\rho dx+\int_{\R^2}|\nabla I_\rho|^2dx-\int_{\R^2}w^{p-1}I^2_\rho dx
      +o\big(\|I_\rho\|^2_2\big)-C\ep^2_\rho\|\nabla I_\rho\|_{L^2}\nonumber\\
      &=(\mathcal{L} I_\rho,I_\rho)+o\big(\|I_\rho\|^2_2\big)-C\ep^2_\rho\|\nabla I_\rho\|_{L^2},
     \end{align}
where the linearized operator $\mathcal{L}$ is defined by
     \begin{equation}\label{eqS3.43}
      \mathcal{L}:=-\Delta-w^{p-1}+1.
     \end{equation}
From \cite [Corollary 11.9 and Theorem 11.8] {LL} and (\ref{eqS3.13}), we deduce that there exists a constant $M>0$ such that
     \begin{equation}\label{eqS3.44}
      (\mathcal{L} I_\rho,I_\rho)\geq M\|I_\rho\|^2_{H^1(\R^2)} \,\ \text{as} \,\ \rho\to\infty.
     \end{equation}
Substituting (\ref{eqS3.44}) into (\ref{eqS3.42}), we have
      \begin{equation}\label{eqS3.45}
      C\epsilon^{2+s}_\rho\geq\f{M}{2}\|I_\rho\|^2_{H^1(\R^2)}+o\big(\|I_\rho\|^2_2\big)-C\ep^2_\rho\|\nabla I_\rho\|_{L^2} \,\ \text{as} \,\ \rho\to\infty.
     \end{equation}
Then, we deduce from (\ref{eqS3.45}) that
      \begin{equation}\label{eqS3.46}
     \|I_\rho\|_{H^1(\R^2)}\leq C\epsilon^{1+\f{s}{2}}_\rho \,\ \text{as} \,\ \rho\to\infty.
     \end{equation}
Applying  (\ref{eqS3.10}), (\ref{A12limibeha-1}) and (\ref{eqS3.46}), one can deduce that
      \begin{equation}\label{eqS3.47}
      \begin{split}
      &\quad\int_{\R^2}x^\perp\cdot(iw_\rho,\nabla w_\rho)dx\\
      &=2\int_{\R^2}R_\rho(x^\perp\cdot\nabla I_\rho)dx
      =2\int_{\R^2}\f{w}{\sqrt{a^*}}(x^\perp\cdot\nabla I_\rho)dx+o(\epsilon^{1+\f{s}{2}}_\rho)\\
      &=-2\int_{\R^2}I_\rho\big(x^\perp\cdot\nabla\f{w}{\sqrt{a^*}} \big)dx+o(\epsilon^{1+\f{s}{2}}_\rho)
      =o(\epsilon^{1+\f{s}{2}}_\rho)\,\ \text{as} \,\ \rho\to\infty,
     \end{split}
     \end{equation}
which yields (\ref{eqS3.38}).
It then follows from (\ref{eqS3.37}) and (\ref{eqS3.38}) that as $\rho\to\infty$,
     \begin{align}\label{eqS3.48}
     &\quad I(\rho)-\hat{I}(\rho)\nonumber\\
      &\geq\int_{\R^2}V_\Om(\ep_\rho x+z_\rho)|w_\rho|^2dx
      +\f{\Om^2}{4}\ep^2_\rho\int_{\R^2}|x|^2|w_\rho|^2dx+o(\epsilon^{1+\f{s}{2}}_\rho)\nonumber\\
      &=
   \left\{
   \begin{array}{lr}
     \f{\ep^s_\rho}{a^*}\big(1+o(1)\big)\ds\int_{\R^2}h(x+\f{z_\rho}{\epsilon_\rho})w^2(x)dx
      &\text{if} \,\ 1<s<2,\\
     \f{\ep^2_\rho}{a^*}\big(1+o(1)\big)\ds\int_{\R^2}\Big(\f{\Om^2}{4}|x|^2w^2(x)+h(x+\f{z_\rho}{\epsilon_\rho})w^2(x)\Big)dx
     &\text{if}\,\ \ s=2.
   \end{array}
   \right.
     \end{align}
 Since $V_{\Om}(x)\to\infty$ as $|x|\to\infty$, one can check that $\{\f{z_\rho}{\ep_\rho}\}$ is bounded uniformly in $\rho$. Combining (\ref{eqS3.36}) and (\ref{eqS3.48}), one can verify that, passing to a subsequence if necessary,
     \begin{equation}\label{11.21}
      \lim\limits_{\rho\to\infty}\f{z_\rho}{\ep_\rho}=y_0,
     \end{equation}
Moreover, since the convergence \eqref{11.21} is independent of what subsequence $\{\f{z_\rho}{\ep_\rho}\}$ we choose, hence (\ref{11.21}) holds for whole sequence, which implies that \eqref{A14concentrate.rate} holds true. The proof of Theorem \ref{thm.limibeha} is thus completed.
\qed
\section{Local uniqueness of minimizers}
In this section, we shall focus on the proof of the local uniqueness of minimizers for $I(\rho)$ as $\rho\to\infty$. By contradiction, suppose that there exist two different minimizers $u_{1\rho}$ and $u_{2\rho}$ of $I(\rho)$ as $\rho\to\infty$ in the sense that $u_{1\rho}\not\equiv u_{2\rho}e^{i\theta}$ for any constant phase $\theta=\theta(\rho)\in[0,2\pi)$.
Define for $j=1, 2$,
      \begin{equation}\label{eqS4.18}
      \tilde{u}_{j\rho}(x):=\ep_\rho u_{j\rho}(\ep_\rho (x+y_0))e^{-i\big(\f{\ep^2_\rho\Om}{2}x\cdot y^\perp_0-\varphi_{j\rho}\big)}
      =\tilde{\mathcal{R}}_{j\rho}(x)+i\tilde{\mathcal{I}}_{j\rho}(x),
      \end{equation}
where $\ep_{\rho}$ is given in (\ref{A13epk}), $y_0$ is defined by ($V_2$), $\tilde{\mathcal{R}}_{j\rho}(x)$ and $\tilde{\mathcal{I}}_{j\rho}(x)$ denote the real and imaginary parts of $\tilde{u}_{j\rho}(x)$, respectively, and the constant phase $\varphi_{j\rho}\in[0,2\pi)$ is chosen properly such that
     \begin{equation}\label{eqS4.19}
     \int_{\R^2}w(x)\tilde{\mathcal{I}}_{j\rho}(x)dx=0, \,\ j=1,2.
     \end{equation}
From (\ref{A11udeELE}) and (\ref{eqS4.18}), we obtain that $\tilde{u}_{j\rho}(x)$ satisfies the following equation
    \begin{equation}\label{eqS4.22}
    \begin{split}
    &-\Delta \tilde{u}_{j\rho}+i\ep^2_\rho\Om(x^\perp\cdot\nabla \tilde{u}_{j\rho})+\Big[\f{\ep^4_\rho\Om^2|x|^2}{4}
     +\ep^2_\rho V_\Om(\ep_\rho (x+y_0))\Big]\tilde{u}_{j\rho}\\
     =&\ep^2_\rho\mu_{j\rho}\tilde{u}_{j\rho}+(a^*)^\f{p-1}{2}|\tilde{u}_{j\rho}|^{p-1}\tilde{u}_{j\rho} \,\ \text{in} \,\ \R^2, \,\ j=1,2,
     \end{split}
     \end{equation}
where $\mu_{j\rho}\in\R$ satisfies
    \begin{equation}\label{eqS4.17}
     \mu_{j\rho}=I(\rho)-\f{(p-1)(a^*)^{\f{p-1}{2}}}{(p+1)\ep^2_\rho}\int_{\R^2}|\tilde{u}_{j\rho}|^{p+1}dx.
     \end{equation}
Furthermore, it follows from Lemmas \ref{lem.important} and  \ref{decay-L} the following Proposition.
\begin{prop}\label{prop4.1}
  Under the assumption of Theorem \ref{thm.localuniqueness}, let $\tilde{u}_{j\rho}$ be defined by \eqref{eqS4.18}, then we have
  \begin{enumerate}
\item The function $\tilde{u}_{j\rho}$ satisfies
     \begin{equation}\label{Y5}
     \lim\limits_{\rho\to\infty}\tilde{u}_{j\rho}(x)=\f{w(x)}{\sqrt{a^*}} \,\ \text{strongly in} \,\ H^{1}(\R^2,\C)\cap L^\infty(\R^2,\C),
     \end{equation}
where $a^{\ast}:=\|w\|^{2}_{2}$ and $w$ is the unique positive solution of (\ref{A4wdeequation}).
\item $\mu_{j\rho}\epsilon^2_\rho\to-1$ as $\rho\to\infty$, and $\tilde{u}_{j\rho}$ decays exponentially in the sense that
     \begin{equation}\label{Y6}
      |\tilde{u}_{j\rho}(x)|\leq Ce^{-\f{2}{3}|x|}, \,\ |\nabla \tilde{u}_{j\rho}(x)|\leq Ce^{-\f{1}{2}|x|} \,\ \text{uniformly in} \,\ \R^2 \,\ \text{as} \,\ \rho\to\infty,
     \end{equation}
where $C>0$ is a constant independent of $0<\rho<\infty$.
\end{enumerate}
\end{prop}
Since $u_{1\rho}\not\equiv u_{2\rho}e^{i\theta}$ holds for any constant phase $\theta=\theta(\rho)\in[0,2\pi)$, we define
      \begin{equation}\label{eqS4.28}
      \tilde{\eta}_\rho(x):=\f{\tilde{u}_{2\rho}(x)-\tilde{u}_{1\rho}(x)}
      {\|\tilde{u}_{2\rho}-\tilde{u}_{1\rho}\|_{L^\infty(\R^2)}}
      =\tilde{\eta}_{1\rho}(x)+i\tilde{\eta}_{2\rho}(x),
      \end{equation}
where $\tilde{\eta}_{1\rho}(x)$ and $\tilde{\eta}_{2\rho}(x)$ denote the real and imaginary parts of $\tilde{\eta}_\rho(x)$, respectively. From (\ref{eqS4.22}) and \eqref{eqS4.28}, we derive that $\tilde{\eta}_\rho$ satisfies
      \begin{equation}\label{eqS4.30}
      \begin{split}
      &-\Delta\tilde{\eta}_\rho+i\ep^2_\rho\Om(x^\perp\cdot\nabla\tilde{\eta}_\rho)
      +\ep^2_\rho V_\Om(\ep_\rho (x+y_0))\tilde{\eta}_\rho+\f{\ep^4_\rho\Om^2|x|^2}{4}\tilde{\eta}_\rho\\
      =&\ep^2_\rho\mu_{1\rho}\tilde{\eta}_\rho+(a^*)^\f{p-1}{2}|\tilde{u}_{2\rho}|^{p-1}\tilde{\eta}_\rho-\f{p-1}{p+1}\tilde{u}_{2\rho}\int_{\R^2}\tilde{C}_\rho(x)\big(|\tilde{u}_{2\rho}|^{\f{p+1}{2}}+|\tilde{u}_{1\rho}|^{\f{p+1}{2}}\big)dx\\
     &
     +\tilde{D}_\rho(x)\tilde{u}_{1\rho} \qquad \text{in} \quad\R^2,
     \end{split}
     \end{equation}
where
      \begin{equation}\label{eqS4.31}
      \begin{split}
        \tilde{C}_\rho(x):&=(a^*)^\f{p-1}{2}
      \f{|\tilde{u}_{2\rho}|^{\f{p+1}{2}}-|\tilde{u}_{1\rho}|^{\f{p+1}{2}}}
      {\|\tilde{u}_{2\rho}-\tilde{u}_{1\rho}\|_{L^\infty(\R^2)}}\\
      &=\f{p+1}{4}(a^*)^\f{p-1}{2}\Big[\tilde{\eta}_{1\rho}(\tilde{\mathcal{R}}_{2\rho}+\tilde{\mathcal{R}}_{1\rho})
      +\tilde{\eta}_{2\rho}(\tilde{\mathcal{I}}_{2\rho}+\tilde{\mathcal{I}}_{1\rho})\Big]\\
      &\quad\quad\cdot\int^{1}_{0}\Big[t|\tilde{u}_{2\rho}|^2+(1-t)|\tilde{u}_{1\rho}|^2\Big]^{\f{p-3}{4}}dt,
      \end{split}
      \end{equation}
and
      \begin{equation}\label{eqS4.32}
      \begin{split}
      \tilde{D}_\rho(x):&=(a^*)^\f{p-1}{2}
      \f{|\tilde{u}_{2\rho}|^{p-1}-|\tilde{u}_{1\rho}|^{p-1}}
      {\|\tilde{u}_{2\rho}-\tilde{u}_{1\rho}\|_{L^\infty(\R^2)}}\\
      &=\f{p-1}{2}(a^*)^\f{p-1}{2}\Big[\tilde{\eta}_{1\rho}(\tilde{\mathcal{R}}_{2\rho}+\tilde{\mathcal{R}}_{1\rho})
      +\tilde{\eta}_{2\rho}(\tilde{\mathcal{I}}_{2\rho}+\tilde{\mathcal{I}}_{1\rho})\Big]\\
      &\quad\quad\cdot\int^{1}_{0}\Big[t|\tilde{u}_{2\rho}|^2+(1-t)|\tilde{u}_{1\rho}|^2\Big]^{\f{p-3}{2}}dt.
      \end{split}
      \end{equation}
\subsection{The convergence of $\tilde{\eta}_\rho$ as $\rho\to\infty$}
In this subsection, we shall give the convergence of $\tilde{\eta}_\rho$ in (\ref{eqS4.30}) as $\rho\to\infty$. We first address the following $L^\infty$-uniform estimates of $\tilde{\eta}_\rho$ and $\nabla\tilde{\eta}_\rho$ as $\rho\to\infty$.
     \begin{lem}\label{lem 4.4}
Suppose that $\tilde{\eta}_\rho$ is defined by (\ref{eqS4.28}), then there exists a constant $C>0$, independent of $0<\rho<\infty$, such that
     \begin{equation}\label{eqS4.33}
     |\tilde{\eta}_\rho(x)|\leq Ce^{-\f{2}{3}|x|}, \ |\nabla \tilde{\eta}_\rho(x)|\leq Ce^{-\f{1}{2}|x|} \,\ \text{uniformly in} \,\ \R^2  \,\ \text{as} \,\ \rho\to\infty.
     \end{equation}
     \end{lem}
The proof of Lemma \ref{lem 4.4} is similar to those in \cite [Lemma 3.2] {GLP}, so we omit it here.
\qed

Based on Lemma \ref{lem 4.4}, we now give the following convergence of $\tilde{\eta}_\rho$ as $\rho\to\infty$.
     \begin{lem}\label{lem 4.5}
Suppose $\tilde{\eta}_\rho$ is defined by (\ref{eqS4.28}). Then passing to a subsequence if necessary, there exist some constants $b_0$, $b_1$ and $b_2$ such that as $\rho\to\infty$,
     \begin{equation}\label{eqS4.50}
     \tilde{\eta}_\rho(x)\to b_0\big(w+\f{p-1}{2}x\cdot\nabla w\big)
     +\sum^2_{j=1}b_j\f{\partial w}{\partial x_j}+M_\rho(x) \,\ \text{uniformly in} \,\ \R^2,
     \end{equation}
where $M_\rho(x)$ satisfies that as $\rho\to\infty$,
     \begin{equation}\label{eqS4.51}
     |M_\rho(x)|\leq C_M(\ep_\rho)e^{-\f{1}{2}|x|}, |\nabla M_\rho(x)|\leq C_M(\ep_\rho)e^{-\f{1}{4}|x|} \,\ \text{uniformly in} \,\ \R^2,
     \end{equation}
for some constants $C_M(\ep_\rho)>0$ satisfying $C_M(\ep_\rho)=o(1)$ as $\rho\to\infty$.
     \end{lem}
     \begin{proof}
Firstly, we claim that passing to a subsequence if necessary, there exist $b_0$, $b_1$ and $b_2$ such that as $\rho\to\infty$,
     \begin{equation}\label{eqS4.52}
     \tilde{\eta}_\rho(x)\to b_0\big(w+\f{p-1}{2}x\cdot\nabla w\big)
     +\sum^2_{j=1}b_j\f{\partial w}{\partial x_j} \,\ \text{uniformly in} \,\ C^1_{loc}(\R^2,\C).
     \end{equation}
In fact, we note from Lemma \ref{lem 4.4} that $(x^\perp\cdot\nabla \tilde{\eta}_\rho)$ is bounded uniformly and decays exponentially for sufficiently large $|x|$ as $\rho\to\infty$. By the definition of $\tilde{C}_\rho(x)$ and $\tilde{D}_\rho(x)$, we obtain from Proposition \ref{prop4.1} that there exists a constant $C>0$ such that
      \begin{equation}\label{eqS4.36}
      \|\tilde{D}_\rho(x)\|_{L^\infty(\R^2)}\leq C \,\ \text{and} \,\ \Big|\int_{\R^2}\tilde{C}_\rho(x)\big(|\tilde{u}_{2\rho}|^{\f{p+1}{2}}+|\tilde{u}_{1\rho}|^{\f{p+1}{2}}\big)dx\Big|\leq C.
      \end{equation}
Using the standard elliptic regularity, one then obtain from (\ref{eqS4.30}) and (\ref{eqS4.36}) that $\tilde{\eta}_\rho\in C^{1,\alpha}_{loc}$ and $\|\tilde{\eta}_\rho\|_{C^{1,\alpha}_{loc}}\leq C$ uniformly as $\rho\to\infty$ for some $\alpha\in(0,1)$. By (\ref{eqS4.30}), we obtain from above that passing to a subsequence if necessary,
     \begin{equation}\label{eqS4.53}
     \tilde{\eta}_\rho:=\tilde{\eta}_{1\rho}+i\tilde{\eta}_{2\rho}
     \to\tilde{\eta}_0:=\tilde{\eta}_1+i\tilde{\eta}_2  \,\ \text{uniformly in} \,\ C^1_{loc}(\R^2,\C)\,\ \text{as}\,\ \rho\to\infty ,
     \end{equation}
where $\tilde{\eta}_0$ solves
     \begin{equation*}
     -\Delta\tilde{\eta}_0+\tilde{\eta}_0-(p-1)w^{p-1}\tilde{\eta}_1-w^{p-1}\tilde{\eta}_0
     =-\Big(\f{p-1}{a^*}\int_{\R^2}w^p\tilde{\eta}_1dx\Big)w \,\ \text{in} \,\ \R^2.
     \end{equation*}
This implies that $(\tilde{\eta}_1,\tilde{\eta}_2)$ satisfies
     \begin{equation}\label{eqS4.54}
      \begin{cases}
     \mathcal{N}\tilde{\eta}_1
     =-\big(\f{p-1}{a^*}\int_{\R^2}w^p\tilde{\eta}_1dx\big)w \,\ \text{in} \,\ \R^2,\\
     \mathcal{L}\tilde{\eta}_2
     =0 \,\ \text{in} \,\ \R^2,\\
     \end{cases}
     \end{equation}
where $\mathcal{L}$ is defined in (\ref{eqS3.43}) and the linearized operator $\mathcal{N}$
is defined by
     \begin{equation}\label{eqS4.55}
      \mathcal{N}:=-\Delta-pw^{p-1}+1.
     \end{equation}
Recall from \cite{K,NT} that
     \begin{equation}\label{eqS4.56}
      ker\mathcal{N}=\Big\{\f{\partial w}{\partial x_1},\f{\partial w}{\partial x_2}\Big\}.
     \end{equation}
Furthermore, one can also check that
     \begin{equation}\label{eqS4.57}
      \mathcal{N}\Big(w+\f{p-1}{2}x\cdot\nabla w\Big)=-(p-1)w.
     \end{equation}
Note from (\ref{eqS4.19}) and (\ref{eqS4.28}) that $\int_{\R^2}w(x)\tilde{\eta}_{2\rho}(x)dx=0$, which together with (\ref{eqS4.53}) implies that
     \begin{equation}\label{eqS4.58}
     \int_{\R^2}w(x)\tilde{\eta}_2(x)dx=0.
     \end{equation}
From \cite [Corollary 11.9 and Theorem 11.8] {LL} and (\ref{eqS4.55})-(\ref{eqS4.58}), one then derive that
     \begin{equation*}
     \tilde{\eta}_1=b_0\big(w+\f{p-1}{2}x\cdot\nabla w\big)
     +\sum^2_{j=1}b_j\f{\partial w}{\partial x_j} \,\ \text{and} \,\ \tilde{\eta}_2\equiv0 \,\ \text{in} \,\ \R^2,
     \end{equation*}
which implies the claim (\ref{eqS4.52}) holds true.

On the other hand, using (\ref{A8wdecay}) and (\ref{eqS4.33}), we deduce that for any fixed sufficiently large $R>0$, as $\rho\to\infty$,
     \begin{equation}\label{eqS4.59}
     \Big|\tilde{\eta}_1-\Big[b_0\big(w+\f{p-1}{2}x\cdot\nabla w\big)
     +\sum^2_{j=1}b_j\f{\partial w}{\partial x_j}\Big]\Big|\leq Ce^{-\f{1}{12}R}e^{-\f{1}{2}|x|} \,\ \text{in} \,\ \R^2/ B_R(0),
     \end{equation}
and
     \begin{equation}\label{eqS4.60}
     \Big|\nabla\tilde{\eta}_1-\nabla\Big[b_0\big(w+\f{p-1}{2}x\cdot\nabla w\big)
     +\sum^2_{j=1}b_j\f{\partial w}{\partial x_j}\Big]\Big|\leq Ce^{-\f{1}{4}R}e^{-\f{1}{4}|x|} \,\ \text{in} \,\ \R^2/ B_R(0).
     \end{equation}
Because $R>0$ is arbitrary, combining (\ref{eqS4.52}), (\ref{eqS4.59}) and (\ref{eqS4.60}), one then conclude that (\ref{eqS4.51}) holds and we are done.
     \end{proof}
\subsection{The refined estimate of $\tilde{\eta}_\rho$ in (\ref{eqS4.30})}
In this subsection, the refined estimate of $\tilde{\eta}_\rho$ as $\rho\to\infty$ shall be established in Lemma \ref{lem 4.6}. In order to prove Lemma \ref{lem 4.6}, we shall firstly give the following estimates of the imaginary part $\tilde{\mathcal{I}}_{j\rho}$ for $\tilde{u}_{j\rho}$ as $\rho\to\infty$.
    \begin{lem}\label{lem 4.3}
Under the assumptions of Theorem \ref{thm.localuniqueness}, let $\tilde{\mathcal{I}}_{j\rho}$ be defined in (\ref{eqS4.18}) for $j=1,2$. Then as $\rho\to\infty$,
     \begin{equation}\label{eqS4.23}
     |\tilde{\mathcal{I}}_{j\rho}(x)|\leq C_{j1}(\ep_\rho)e^{-\f{1}{4}|x|}, \,\ |\nabla \tilde{\mathcal{I}}_{j\rho}(x)|\leq C_{j2}(\ep_\rho)e^{-\f{1}{8}|x|} \,\ \text{uniformly in} \,\ \R^2,
     \end{equation}
where the constants $C_{j1}(\ep_\rho), C_{j2}(\ep_\rho)>0$ satisfy
     \begin{equation}\label{eqS4.24}
     C_{j1}(\ep_\rho)=o(\ep^2_\rho) \,\ \text{and} \,\  C_{j2}(\ep_\rho)=o(\ep^2_\rho) \,\ \text{as} \,\ \rho\to\infty.
     \end{equation}
     \end{lem}
     \begin{proof}
     Note from \eqref{eqS4.18} and \eqref{eqS4.22} that the imaginary part $\tilde{\mathcal{I}}_{j\rho}$ of $\tilde{u}_{j\rho}$ satisfies
     \begin{equation}\label{eqS4.3}
     \mathcal{L}_\rho \tilde{\mathcal{I}}_{j\rho}(x)=-\ep^2_\rho\Om(x^\perp\cdot\nabla \tilde{\mathcal{R}}_{j\rho}) \,\ \text{in} \,\ \R^2, \,\ \int_{\R^2}\tilde{\mathcal{I}}_{j\rho}(x)w(x)dx=0,
     \end{equation}
where $\mathcal{L}_\rho$ is defined as
     \begin{equation}\label{eqS4.4}
     \mathcal{L}_\rho:=-\Delta+\f{\ep^4_\rho\Om^2|x|^2}{4}
     +\ep^2_\rho V_\Om(\ep_\rho (x+y_0))-\ep^2_\rho\mu_{j\rho}-(a^*)^\f{p-1}{2}|\tilde{u}_{j\rho}|^{p-1}.
     \end{equation}
Applying (\ref{A8wdecay}), (\ref{Y5}) and (\ref{Y6}), we deduce that as $\rho\to\infty$,
     \begin{equation}\label{eqS4.5}
     \big|\ep^2_\rho\Om(x^\perp\cdot\nabla \tilde{\mathcal{R}}_{j\rho}\big|\leq C(\ep_\rho)e^{-\f{1}{4}|x|} \,\ \text{uniformly in} \,\ \R^2,
     \end{equation}
where the constant $C(\ep_\rho)>0$ satisfies $C(\ep_\rho)=o(\ep^2_\rho)$ as $\rho\to\infty$.

Based on (\ref{eqS4.5}), we next prove \eqref{eqS4.23} and \eqref{eqS4.24}. Multiplying (\ref{eqS4.3}) by $\tilde{\mathcal{I}}_{j\rho}$ and integrating over $\R^2$, using the H\"{o}lder inequality, we obtain from (\ref{eqS4.5}) that
     \begin{equation}\label{eqS4.8}
     \int_{\R^2}(\mathcal{L}_\rho \tilde{\mathcal{I}}_{j\rho})\tilde{\mathcal{I}}_{j\rho} dx=-\ep^2_\rho\Om\int_{\R^2}(x^\perp\cdot\nabla \tilde{\mathcal{R}}_{j\rho})\tilde{\mathcal{I}}_{j\rho} dx=o(\ep^2_\rho)\|\tilde{\mathcal{I}}_{j\rho}\|_{L^2(\R^2)} \,\ \text{as} \,\ \rho\to\infty.
     \end{equation}
From (\ref{eqS3.44}), \eqref{eqS4.19} and \eqref{Y5}, we derive that as $\rho\to\infty$,
     \begin{align}\label{eqS4.9}
     \int_{\R^2}(\mathcal{L}_\rho \tilde{\mathcal{I}}_{j\rho})\tilde{\mathcal{I}}_{j\rho} dx&\geq
     \int_{\R^2}\Big[(\mathcal{L}\tilde{\mathcal{I}}_{j\rho})\tilde{\mathcal{I}}_{j\rho}
     -(1+\ep^2_\rho\mu_{j\rho})\tilde{\mathcal{I}}^2_{j\rho}
     -\big((a^*)^{\f{p-1}{2}}|\tilde{u}_{j\rho}|^{p-1}-w^{p-1}\big)\tilde{\mathcal{I}}^2_{j\rho}\Big]dx\nonumber\\
     &=\int_{\R^2}(\mathcal{L}\tilde{\mathcal{I}}_{j\rho})\tilde{\mathcal{I}}_{j\rho} dx+o(1)\int_{\R^2}\tilde{\mathcal{I}}^2_{j\rho} dx\geq\f{M}{2}\|\tilde{\mathcal{I}}_{j\rho}\|^2_{H^1(\R^2)},
     \end{align}
where $\mathcal{L}$ is defined in (\ref{eqS3.43}) and $M>0$, which is independent of $0<\rho<\infty$, is given by (\ref{eqS3.44}). It then follows from (\ref{eqS4.8}) and (\ref{eqS4.9}) that
     \begin{equation}\label{eqS4.10}
     \|\tilde{\mathcal{I}}_{j\rho}\|_{H^1(\R^2)}=o(\ep^2_\rho) \,\ \text{as} \,\ \rho\to\infty.
     \end{equation}
On the other hand, we deduce from (\ref{eqS4.3}) that $|\tilde{\mathcal{I}}_{j\rho}|^2$ satisfies the following equation
     \begin{align*}
     &\Big[-\f{1}{2}\Delta+\Big(\f{\ep^4_\rho\Om^2|x|^2}{4}
     +\ep^2_\rho V_\Om\big(\ep_\rho (x+y_0)\big)-\ep^2_\rho\mu_{j\rho}-(a^*)^\f{p-1}{2}|\tilde{u}_{j\rho}|^{p-1}\Big)\Big]
     |\tilde{\mathcal{I}}_{j\rho}|^2
     +|\nabla \tilde{\mathcal{I}}_{j\rho}|^2\\
     =&-\ep^2_\rho\Om(x^\perp\cdot\nabla \tilde{\mathcal{R}}_{j\rho})\tilde{\mathcal{I}}_{j\rho}  \,\ \text{in} \,\ \R^2,
     \end{align*}
which yields that
     \begin{equation}\label{eqS4.11}
     -\f{1}{2}\Delta|\tilde{\mathcal{I}}_{j\rho}|^2-\ep^2_\rho\mu_{j\rho}
     |\tilde{\mathcal{I}}_{j\rho}|^2-(a^*)^\f{p-1}{2}|\tilde{u}_{j\rho}|^{p-1}
     |\tilde{\mathcal{I}}_{j\rho}|^2
     \leq-\ep^2_\rho\Om(x^\perp\cdot\nabla \tilde{\mathcal{R}}_{j\rho})\tilde{\mathcal{I}}_{j\rho}  \,\ \text{in} \,\ \R^2.
     \end{equation}
Since $\mu_{j\rho}\epsilon^2_\rho\to-1$ as $\rho\to\infty$, by De Giorgi-Nash-Moser theory \cite [Theorem 4.1]{HL}, we deduce from (\ref{eqS4.11}) that for any $\xi\in\R^2$,
     \begin{equation}\label{eqS4.12}
     \sup_{x\in B_1(\xi)}|\tilde{\mathcal{I}}_{j\rho}(x)|^2\leq C\Big(\|\tilde{\mathcal{I}}^2_{j\rho}\|_{L^2(B_2(\xi))}+\|\ep^2_\rho\Om(x^\perp\cdot\nabla \tilde{\mathcal{R}}_{j\rho})\tilde{\mathcal{I}}_{j\rho}\|_{L^{\f{2}{p-1}}(B_2(\xi))}\Big).
     \end{equation}
Employing Proposition \ref{prop4.1}, we then derive from (\ref{eqS4.5}), (\ref{eqS4.10}) and (\ref{eqS4.12}) that
     \begin{equation}\label{eqS4.13}
     \|\tilde{\mathcal{I}}_{j\rho}\|_{L^\infty(\R^2)}=o(\ep^2_\rho)\,\ \text{as} \,\ \rho\to\infty,
     \end{equation}
and hence
     \begin{equation}\label{eqS4.14}
     \big|(a^*)^\f{p-1}{2}|\tilde{u}_{j\rho}|^{p-1}
     \tilde{\mathcal{I}}_{j\rho}-\ep^2_\rho\Om(x^\perp\cdot\nabla \tilde{\mathcal{R}}_{j\rho})\big|\leq C_0(\ep_\rho)e^{-\f{1}{4}|x|} \,\ \text{uniformly in} \,\ \R^2,
     \end{equation}
where the constant $C_0(\ep_\rho)>0$ satisfies $C_0(\ep_\rho)=o(\ep^2_\rho)$ as $\rho\to\infty$. By the comparison principle, we thus deduce from (\ref{eqS4.11}), (\ref{eqS4.13}) and (\ref{eqS4.14}) that
     \begin{equation}\label{eqS4.15}
     |\tilde{\mathcal{I}}_{j\rho}(x)|\leq C_{j1}(\ep_\rho)e^{-\f{1}{4}|x|} \,\ \text{uniformly in} \,\ \R^2 \,\ \text{as} \,\ \rho\to\infty,
     \end{equation}
where the constant $C_{j1}(\ep_\rho)>0$ satisfies $C_{j1}(\ep_\rho)=o(\ep^2_\rho)$ as $\rho\to\infty$. Furthermore, by the exponential decay \eqref{eqS4.15}, applying gradient estimates (see (3.15) in \cite {GT}) to (\ref{eqS4.3}) then yields that the gradient estimate of \eqref{eqS4.23} and \eqref{eqS4.24} hold true. The proof of Lemma \ref{lem 4.3} is thus completed.
     \end{proof}
Combining Lemma \ref{lem 4.5} and Lemma \ref{lem 4.3}, we now establish the refined estimates of $\tilde{\eta}_\rho$ as $\rho\to\infty$.
     \begin{lem}\label{lem 4.6}
Suppose $\{\tilde{\eta}_\rho\}$ is the sequence obtained in Lemma \ref{lem 4.5}. Then the imaginary $\tilde{\eta}_{2\rho}$ of $\tilde{\eta}_\rho$ satisfies that as $\rho\to\infty$,
     \begin{equation}\label{eqS4.61}
     \tilde{\eta}_{2\rho}(x)=\f{\ep^2_\rho\Om}{2}(-b_1x_2+b_2x_1)w(x)+E_\rho(x) \,\ \text{uniformly in} \,\ \R^2,
     \end{equation}
where $(x_1,x_2)=x\in\R^2$, $b_1$ and $b_2$ are as in Lemma \ref{lem 4.5}, and $E_\rho(x)$ satisfies that as $\rho\to\infty$,
     \begin{equation}\label{eqS4.62}
     |E_\rho(x)|\leq C_E(\ep_\rho)e^{-\f{1}{8}|x|}, |\nabla E_\rho(x)|\leq C_E(\ep_\rho)e^{-\f{1}{16}|x|} \,\ \text{uniformly in} \,\ \R^2
     \end{equation}
for some constants $C_E(\ep_\rho)>0$ satisfying $C_E(\ep_\rho)=o(\ep^2_\rho)$ as $\rho\to\infty$.
     \end{lem}
     \begin{proof}
We first obtain from (\ref{eqS4.19}) and (\ref{eqS4.30}) that
      \begin{equation}\label{eqS4.63}
      \mathcal{L}_{2\rho}\tilde{\eta}_{2\rho}
      =-\ep^2_\rho\Om(x^\perp\cdot\nabla\tilde{\eta}_{1\rho})+G_\rho(x)
       \,\ \text{in} \,\ \R^2, \,\ \int_{\R^2}\tilde{\eta}_{2\rho}wdx=0,
     \end{equation}
where $\mathcal{L}_{j\rho}$ is defined for $j=1,2$,
     \begin{equation}\label{eqS4.27}
     \mathcal{L}_{j\rho}:=-\Delta+\f{\ep^4_\rho\Om^2|x|^2}{4}
     +\ep^2_\rho V_\Om(\ep_\rho (x+y_0))-\ep^2_\rho\mu_{j\rho}-(a^*)^\f{p-1}{2}|\tilde{u}_{j\rho}|^{p-1},
     \end{equation}
and $G_\rho(x)$ is defined as
     \begin{equation*}
     G_\rho(x):=-\f{p-1}{p+1}\tilde{\mathcal{I}}_{2\rho}\int_{\R^2}\tilde{C}_\rho(x)\big(|\tilde{u}_{2\rho}|^{\f{p+1}{2}}+|\tilde{u}_{1\rho}|^{\f{p+1}{2}}\big)dx
      +\tilde{D}_\rho(x)\tilde{\mathcal{I}}_{1\rho},
     \end{equation*}
where $\tilde{C}_\rho$ and $\tilde{D}_\rho$ are defined in (\ref{eqS4.31}) and (\ref{eqS4.32}). It then follows from Lemma \ref{lem 4.5} and (\ref{eqS4.63}) that
      \begin{equation}\label{eqS4.64}
      \mathcal{L}_{2\rho}\tilde{\eta}_{2\rho}
      =-\ep^2_\rho\Om\Big[-b_1\f{\partial w}{\partial x_2}+b_2\f{\partial w}{\partial x_1}+Re(x^\perp\cdot\nabla M_\rho)\Big]+G_\rho(x)
       \,\ \text{in} \,\ \R^2, \,\ \int_{\R^2}\tilde{\eta}_{2\rho}wdx=0,
     \end{equation}
where $Re(\cdot)$ denotes the real part.
On the other hand, one can check that $-\f{1}{2}x_jw$ is the unique solution of the following equation
      \begin{equation}\label{eqS4.65}
      \mathcal{L}u=\f{\partial w}{\partial x_j}, \,\ \int_{\R^2}uwdx=0, \,\ j=1,2,
     \end{equation}
where $\mathcal{L}$ is defined in (\ref{eqS3.43}). Denote
      \begin{equation*}
      E_\rho(x):=\tilde{\eta}_{2\rho}(x)-\f{\ep^2_\rho\Om}{2}(-b_1x_2+b_2x_1)w(x).
     \end{equation*}
By (\ref{eqS4.65}), we then derive from (\ref{eqS4.64}) that $E_\rho(x)$ satisfies
     \begin{equation}\label{eqS4.66}
     \mathcal{L}_{2\rho}E_\rho(x)=\f{\ep^2_\rho\Om}{2}(\mathcal{L}-\mathcal{L}_{2\rho})\big[(-b_1x_2+b_2x_1)w(x)\big]
     -\ep^2_\rho\Om Re(x^\perp\cdot\nabla M_\rho)+G_\rho(x) \,\ \text{in} \,\ \R^2,
     \end{equation}
and
      \begin{equation}\label{eqS4.67}
      \int_{\R^2}E_\rho(x)wdx=0.
      \end{equation}
\indent Now, we shall estimate the right hand side of (\ref{eqS4.66}). From the definition of $\mathcal{L}_{2\rho}$ in (\ref{eqS4.27}), we obtain that
      \begin{equation*}\label{eqS4.68}
     \f{\ep^2_\rho\Om}{2}\Big|(\mathcal{L}-\mathcal{L}_{2\rho})\big[(-b_1x_2+b_2x_1)w(x)\big]\Big|
     \leq C(\ep_\rho)e^{-\f{1}{4}|x|} \,\ \text{uniformly in} \,\ \R^2 \,\ \text{as} \,\ \rho\to\infty,
     \end{equation*}
where $C(\ep_\rho)>0$ satisfies
     \begin{equation}\label{eqS4.69}
     C(\ep_\rho)=o(\ep^2_\rho) \,\ \text{as} \,\ \rho\to\infty.
     \end{equation}
By Lemma \ref{lem 4.5}, we obtain that
     \begin{equation}\label{eqS4.70}
     \Big| Re(\ep^2_\rho\Om x^\perp\cdot\nabla M_\rho)\Big|\leq C(\ep_\rho)e^{-\f{1}{8}|x|} \,\ \text{uniformly in} \,\ \R^2 \,\ \text{as} \,\ \rho\to\infty,
     \end{equation}
where $C(\ep_\rho)>0$ also satisfies (\ref{eqS4.69}). Furthermore, we derive from Lemma \ref{lem 4.3} and (\ref{eqS4.36}) that
     \begin{equation}\label{eqS4.71}
     |G_\rho(x)|\leq C(\ep_\rho)e^{-\f{1}{4}|x|} \,\ \text{uniformly in} \,\ \R^2 \,\ \text{as} \,\ \rho\to\infty,
     \end{equation}
where $C(\ep_\rho)>0$ also satisfies (\ref{eqS4.69}). Applying the above estimates, we then derive from (\ref{eqS4.66}) that
     \begin{align}\label{eqS4.72}
     \big|\mathcal{L}_{2\rho}E_\rho(x)\big|
     &=\Big|\f{\ep^2_\rho\Om}{2}(\mathcal{L}-\mathcal{L}_{2\rho})\big[(-b_1x_2+b_2x_1)w(x)\big]
     -\ep^2_\rho\Om Re(x^\perp\cdot\nabla M_\rho)+G_\rho(x)\Big|\nonumber\\
     &\leq C(\ep_\rho)e^{-\f{1}{8}|x|} \,\ \text{uniformly in} \,\ \R^2 \,\ \text{as} \,\ \rho\to\infty.
     \end{align}
Similar to the argument of proving Lemma \ref{lem 4.3}, one can conclude from (\ref{eqS4.66}), (\ref{eqS4.67}) and (\ref{eqS4.72}) that (\ref{eqS4.61}) and (\ref{eqS4.62}) hold true. The proof of Lemma \ref{lem 4.6} is thus completed.
     \end{proof}
In the following, we shall follow the refined estimates of $\tilde{\eta}_\rho$ to complete the proof of Theorem \ref{thm.localuniqueness} on the local uniqueness of minimizers for $I(\rho)$ as $\rho\to\infty$.
\vskip 0.2truein
\noindent{\bf{Proof of Theorem 1.3:}}  Argue by contradiction, we assume that, up to a constant phase, there exist two different minimizers $u_{1\rho}$ and $u_{2\rho}$ of $I(\rho)$ as $\rho\to\infty$, that is, $u_{1\rho}\not\equiv u_{2\rho}e^{i\theta}$ for any constant phase $\theta=\theta(\rho)\in[0, 2\pi)$. Recall from \eqref{eqS4.22} and \eqref{eqS4.17} that
$\tilde{u}_{j\rho}(x):=\tilde{\mathcal{R}}_{j\rho}(x)+i\tilde{\mathcal{I}}_{j\rho}(x)$ defined in \eqref{eqS4.18} satisfies the following equation
    \begin{equation}\label{S1}
    \begin{split}
    &-\Delta \tilde{u}_{j\rho}+i\ep^2_\rho\Om(x^\perp\cdot\nabla \tilde{u}_{j\rho})+\Big[\f{\ep^4_\rho\Om^2|x|^2}{4}
     +\ep^2_\rho V_\Om(\ep_\rho (x+y_0))\Big]\tilde{u}_{j\rho}\\
     =&\ep^2_\rho\mu_{j\rho}\tilde{u}_{j\rho}+(a^*)^\f{p-1}{2}|\tilde{u}_{j\rho}|^{p-1}\tilde{u}_{j\rho} \,\ \text{in} \,\ \R^2, \,\ j=1,2,
     \end{split}
     \end{equation}
where $\mu_{j\rho}\in\R$ satisfies
    \begin{equation}\label{S2}
     \mu_{j\rho}=I(\rho)-\f{(p-1)(a^*)^{\f{p-1}{2}}}{(p+1)\ep^2_\rho}\int_{\R^2}|\tilde{u}_{j\rho}|^{p+1}dx.
     \end{equation}
Then, the real part $\tilde{\mathcal{R}}_{j\rho}$ of $\tilde u_{j\rho}$ satisfies
    \begin{equation}\label{eqS4.76}
    \begin{split}
    &-\Delta\tilde{\mathcal{R}}_{j\rho}
    -\ep^2_\rho\Om (x^\perp\cdot\nabla \tilde{\mathcal{I}}_{j\rho})+\Big[\f{\ep^4_\rho\Om^2|x|^2}{4}
     +\ep^2_\rho V_\Om(\ep_\rho(x+y_0))\Big]\tilde{\mathcal{R}}_{j\rho}\\
     =&\ep^2_\rho\mu_{j\rho}\tilde{\mathcal{R}}_{j\rho}
     +(a^*)^\f{p-1}{2}|\tilde{u}_{j\rho}|^{p-1}\tilde{\mathcal{R}}_{j\rho} \,\ \text{in} \,\ \R^2, \,\ j=1,2.
     \end{split}
     \end{equation}

To obtain a contradiction, we next prove the constants $b_i=0$ ($i=0, 1, 2$) defined in (\ref{eqS4.50}) by constructing Pohozaev identities of $\tilde{\mathcal{R}}_{j\rho}$, where $j=1, 2$. This process is divided into the following three steps:
\vskip0.2cm
\noindent $\mathbf{Step\,\ 1}$. We claim that the following Pohozaev identity holds
     \begin{equation}\label{eqS4.73}
     b_0\f{p-1}{2}\int_{\R^2}\f{\partial h(x+y_0)}{\partial x_l}(x\cdot\nabla w^2)dx
     +\sum^2_{j=1}b_j\int_{\R^2}\f{\partial h(x+y_0)}{\partial x_l}\f{\partial w^2}{\partial x_j}dx=0, \,\ l=1,2,
     \end{equation}
where $b_0, b_1$, and $b_2$ are defined in (\ref{eqS4.50}).

In order to prove the claim \eqref{eqS4.73}, we first multiply (\ref{eqS4.76}) by $\f{\partial\tilde{\mathcal{R}}_{j\rho}}{\partial x_l}$ and integrating over $\R^2$, where $j=1,2$ and $l=1,2$, one then obtain
    \begin{equation}\label{eqS4.77}
    \begin{split}
    &-\int_{\R^2}\Delta\tilde{\mathcal{R}}_{j\rho}
    \f{\partial\tilde{\mathcal{R}}_{j\rho}}{\partial x_l}
    -\int_{\R^2}\ep^2_\rho\Om(x^\perp\cdot\nabla \tilde{\mathcal{I}}_{j\rho})
    \f{\partial\tilde{\mathcal{R}}_{j\rho}}{\partial x_l}\\
    +&\int_{\R^2}\f{1}{2}\Big[\f{\ep^4_\rho\Om^2|x|^2}{4}
     +\ep^2_\rho V_\Om(\ep_\rho(x+y_0))\Big]
     \f{\partial|\tilde{\mathcal{R}}_{j\rho}|^2}{\partial x_l}\\
     =&\int_{\R^2}\f{1}{2}\ep^2_\rho\mu_{j\rho}
     \f{\partial|\tilde{\mathcal{R}}_{j\rho}|^2}{\partial x_l}
     +\int_{\R^2}\f{(a^*)^\f{p-1}{2}}{2}|\tilde{u}_{j\rho}|^{p-1}
     \f{\partial|\tilde{\mathcal{R}}_{j\rho}|^2}{\partial x_l}.
     \end{split}
     \end{equation}
By the exponential decay \eqref{Y6}, we calculate that for $j=1, 2$,
    \begin{equation*}
    \begin{split}
    -\int_{\R^2}\Delta\tilde{\mathcal{R}}_{j\rho}
    \f{\partial\tilde{\mathcal{R}}_{j\rho}}{\partial x_l}
    &=-\lim\limits_{\R\to\infty}\int_{B_R(0)}\Delta\tilde{\mathcal{R}}_{j\rho}
    \f{\partial\tilde{\mathcal{R}}_{j\rho}}{\partial x_l}\\
    &=-\lim\limits_{\R\to\infty}\int_{\partial B_R(0)}
    \Big(\f{\partial\tilde{\mathcal{R}}_{j\rho}}{\partial\nu}
    \f{\partial\tilde{\mathcal{R}}_{j\rho}}{\partial x_l}
    -\f{1}{2}|\nabla\tilde{\mathcal{R}}_{j\rho}|^2\nu_l\Big)dS=0,
    \end{split}
    \end{equation*}
where $\nu=(\nu_1,\nu_2)$ denotes the outward unit of $\partial B_R(0)$, and
    \begin{equation*}
    \int_{\R^2}\f{1}{2}\Big[\f{\ep^4_\rho\Om^2|x|^2}{4}
     +\ep^2_\rho V_\Om(\ep_\rho(x+y_0))\Big]\f{\partial|\tilde{{\mathcal{R}}}_{j\rho}|^2}{\partial x_l}
    =-\int_{\R^2}\f{1}{2}\Big[\f{\ep^4_\rho\Om^2x_l}{2}
     +\ep^2_\rho\f{\partial V_\Om(\ep_\rho(x+y_0))}{\partial x_l}\Big]|\tilde{\mathcal{R}}_{j\rho}|^2,
     \end{equation*}
where $(x_1,x_2)=x\in\R^2$. It then follows from (\ref{eqS4.77}) that
    \begin{equation}\label{eqS4.81}
    \begin{split}
    &\int_{\R^2}\f{1}{2}\Big[\f{\ep^4_\rho\Om^2x_l}{2}
     +\ep^2_\rho\f{\partial V_\Om(\ep_\rho(x+y_0))}{\partial x_l}\Big]|\tilde{\mathcal{R}}_{j\rho}|^2\\
     =&-\int_{\R^2}\ep^2_\rho\Om(x^\perp\cdot\nabla \tilde{\mathcal{I}}_{j\rho})
    \f{\partial\tilde{\mathcal{R}}_{j\rho}}{\partial x_l}
    -\int_{\R^2}\f{(a^*)^\f{p-1}{2}}{2}|\tilde{u}_{j\rho}|^{p-1}
     \f{\partial|\tilde{\mathcal{R}}_{j\rho}|^2}{\partial x_l},\,\ j=1, 2.
     \end{split}
     \end{equation}
Moreover, using \eqref{eqS4.18}, \eqref{eqS4.28} and \eqref{eqS4.81}, we have
    \begin{align}\label{eqS4.82}
    &\int_{\R^2}\f{\ep^2_\rho}{2}
    \f{\partial V_\Om(\ep_\rho(x+y_0))}{\partial x_l}(\tilde{\mathcal{R}}_{2\rho}
    +\tilde{\mathcal{R}}_{1\rho})\tilde{\eta}_{1\rho}\nonumber\\
    =&-\int_{\R^2}\f{\ep^4_\rho\Om^2x_l}{4}
    (\tilde{\mathcal{R}}_{2\rho}+\tilde{\mathcal{R}}_{1\rho})\tilde{\eta}_{1\rho}
     \nonumber\\
     &-\int_{\R^2}\Big\{\ep^2_\rho\Om (x^\perp\cdot\nabla \tilde{\eta}_{2\rho})
     \f{\partial\tilde{R}_{2\rho}}{\partial x_l}
     +\ep^2_\rho\Om(x^\perp\cdot\nabla\tilde{\mathcal{I}}_{1\rho})
     \f{\partial\tilde{\eta}_{1\rho}}{\partial x_l}\Big\}
     \\
     &+\int_{\R^2}\f{(a^*)^\f{p-1}{2}}{2}|\tilde{u}_{2\rho}|^{p-1}
     \f{\partial\big[\tilde{\eta}_{2\rho}(\tilde{\mathcal{I}}_{2\rho}+\tilde{\mathcal{I}}_{1\rho})\big]}{\partial x_l}\nonumber\\
     &+\int_{\R^2}\f{(p-1)(a^*)^\f{p-1}{2}}{4}
     \Big[\tilde{\eta}_{1\rho}(\tilde{\mathcal{R}}_{2\rho}+\tilde{\mathcal{R}}_{1\rho})
      +\tilde{\eta}_{2\rho}(\tilde{\mathcal{I}}_{2\rho}+\tilde{\mathcal{I}}_{1\rho})\Big]\nonumber\\
      &\quad\quad\cdot\int^{1}_{0}\Big[t|\tilde{u}_{2\rho}|^2+(1-t)|\tilde{u}_{1\rho}|^2\Big]^{\f{p-3}{2}}dt
     \f{\partial|\tilde{\mathcal{I}}_{1\rho}|^2}{\partial x_l}.
     \nonumber
     \end{align}

Next, we shall prove the claim \eqref{eqS4.73} by estimating all terms of (\ref{eqS4.82}). Applying Lemma \ref{lem 4.5} and \eqref{eqS4.65}, we deduce that as $\rho\to\infty$,
    \begin{equation}\label{eqS4.84}
    \begin{split}
    &-\int_{\R^2}\f{\ep^4_\rho\Om^2x_l}{4}
    (\tilde{\mathcal{R}}_{2\rho}+\tilde{\mathcal{R}}_{1\rho})\tilde{\eta}_{1\rho}\\
    =&-\f{\ep^4_\rho\Om^2}{2\sqrt{a^*}}
    \int_{\R^2}x_lw\Big[b_0\big(w+\f{p-1}{2}x\cdot\nabla w\big)
     +\sum^2_{j=1}b_j\f{\partial w}{\partial x_j}\Big]+o(\ep^4_\rho)\\
    =&-\f{\ep^4_\rho\Om^2}{4\sqrt{a^*}}
    \int_{\R^2}b_lx_l\f{\partial w^2}{\partial x_l}+o(\ep^4_\rho)\\
    =&\f{\ep^4_\rho\Om^2\sqrt{a^*}b_l}{4}+o(\ep^4_\rho), \,\ l=1,2.
    \end{split}
    \end{equation}
Similarly, we obtain from Lemma \ref{lem 4.6} that as $\rho\to\infty$,
    \begin{equation}\label{eqS4.85}
    \begin{split}
    &-\int_{\R^2}\ep^2_\rho\Om \big(x^\perp\cdot\nabla \tilde{\eta}_{2\rho}\big)
    \f{\partial\tilde{\mathcal{R}}_{2\rho}}{\partial x_l}\\
    =&-\ep^2_\rho\Om\int_{\R^2} x^\perp\cdot\nabla \Big[\f{\ep^2_\rho\Om}{2}(-b_1x_2+b_2x_1)w\Big]
    \f{\partial\big(\f{w}{\sqrt{a^*}}\big)}{\partial x_l}+o(\ep^4_\rho)\\
    =&-\f{\ep^4_\rho\Om^2}{2\sqrt{a^*}}
    \int_{\R^2}(-b_2x_2w-b_1x_1w)\f{\partial w}{\partial x_l}+o(\ep^4_\rho)\\
    =&-\f{\ep^4_\rho\Om^2\sqrt{a^*}b_l}{4}+o(\ep^4_\rho), \,\ l=1,2.
    \end{split}
    \end{equation}
From Proposition \ref{prop4.1}, Lemmas \ref{lem 4.3} and \ref{lem 4.6}, we also derive that as $\rho\to\infty$,
    \begin{equation}\label{eqS4.86}
    \int_{\R^2}\ep^2_\rho\Om (x^\perp\cdot\nabla\tilde{\mathcal{I}}_{1\rho})
    \f{\partial\tilde{\eta}_{1\rho}}{\partial x_l}=o(\ep^4_\rho),
    \end{equation}
    \begin{equation}\label{eqS4.87}
    \int_{\R^2}\f{(a^*)^\f{p-1}{2}}{2}|\tilde{u}_{2\rho}|^{p-1}
     \f{\partial\big[\tilde{\eta}_{2\rho}(\tilde{\mathcal{I}}_{2\rho}+\tilde{\mathcal{I}}_{1\rho})\big]}{\partial x_l}
    =o(\ep^4_\rho),
    \end{equation}
and
    \begin{align}\label{eqS4.88}
    &\int_{\R^2}\f{(p-1)(a^*)^\f{p-1}{2}}{4}
     \Big[\tilde{\eta}_{1\rho}(\tilde{\mathcal{R}}_{2\rho}+\tilde{\mathcal{R}}_{1\rho})
      +\tilde{\eta}_{2\rho}(\tilde{\mathcal{I}}_{2\rho}+\tilde{\mathcal{I}}_{1\rho})\Big]\nonumber\\
      &\quad\quad\cdot\int^{1}_{0}\Big[t|\tilde{u}_{2\rho}|^2+(1-t)|\tilde{u}_{1\rho}|^2\Big]^{\f{p-3}{2}}dt
     \f{\partial|\tilde{\mathcal{I}}_{1\rho}|^2}{\partial x_l}=o(\ep^4_\rho).\\
     \nonumber
    \end{align}
Employing (\ref{eqS4.84})-(\ref{eqS4.88}), $(V_1)$, $(V_2)$ and Lemma \ref{lem 4.5}, one can deduce from (\ref{eqS4.82}) that for $l=1, 2$, as $\rho\to\infty$,
    \begin{align}\label{eqS4.94}
    o(\ep^4_\rho)=&\int_{\R^2}\f{\ep^2_\rho}{2}
    \f{\partial V_\Om(\ep_\rho(x+y_0))}{\partial x_l}
    (\tilde{\mathcal{R}}_{2\rho}+\tilde{\mathcal{R}}_{1\rho})
    \tilde{\eta}_{1\rho}dx\nonumber\\
    =&\f{\ep^{2+s}_\rho}{2}\int_{\R^2}
    \f{\partial h(x+y_0)}{\partial x_l}
    (\tilde{\mathcal{R}}_{2\rho}+\tilde{\mathcal{R}}_{1\rho})
    \tilde{\eta}_{1\rho}dx+o(\ep^{2+s}_\rho)\nonumber\\
    =&\f{\ep^{2+s}_\rho}{\sqrt{a^*}}\int_{\R^2}
    \f{\partial h(x+y_0)}{\partial x_l}w
    \Big[b_0\big(w+\f{p-1}{2}x\cdot\nabla w\big)
     +\sum^2_{j=1}b_j\f{\partial w}{\partial x_j}\Big]dx+o(\ep^{2+s}_\rho)\\
    =&\f{\ep^{2+s}_\rho}{2\sqrt{a^*}}\int_{\R^2}
    \Big[b_0\f{p-1}{2}\f{\partial h(x+y_0)}{\partial x_l}(x\cdot\nabla w^2)
    +\sum^2_{j=1}b_j\f{\partial h(x+y_0)}{\partial x_l}
    \f{\partial w^2}{\partial x_j}\Big]dx+o(\ep^{2+s}_\rho),\nonumber
    \end{align}
which implies that the claim (\ref{eqS4.73}) holds true.
\vskip0.2cm
\noindent $\mathbf{Step\,\ 2}$. The coefficient $b_0=0$ in (\ref{eqS4.50}).

Multiplying (\ref{eqS4.76}) by $(x\cdot\nabla\tilde{\mathcal{R}}_{j\rho})$ and integrating over $\R^2$, where $j=1, 2$, we have
    \begin{equation}\label{eqS4.95}
    \begin{split}
    &-\int_{\R^2}\Delta\tilde{R}_{j\rho}(x\cdot\nabla\tilde{\mathcal{R}}_{j\rho})
    -\int_{\R^2}\ep^2_\rho\Om(x^\perp\cdot\nabla \tilde{\mathcal{I}}_{j\rho})
    (x\cdot\nabla\tilde{\mathcal{R}}_{j\rho})\\
    =&\int_{\R^2}\Big[\ep^2_\rho\mu_{j\rho}
    -\f{\ep^4_\rho\Om^2|x|^2}{4}-\ep^2_\rho V_\Om(\ep_\rho(x+y_0))\Big]\tilde{\mathcal{R}}_{j\rho}
    (x\cdot\nabla\tilde{\mathcal{R}}_{j\rho})\\
    &+(a^*)^\f{p-1}{2}\int_{\R^2}|\tilde{u}_{j\rho}|^{p-1}\tilde{\mathcal{R}}_{j\rho}
    (x\cdot\nabla\tilde{\mathcal{R}}_{j\rho}).
    \end{split}
    \end{equation}
By the integration by parts, we derive from \eqref{Y6} that for $j=1, 2$,
    \begin{align*}
    A_{j\rho}:=&-\int_{\R^2}\Delta\tilde{\mathcal{R}}_{j\rho}(x\cdot\nabla\tilde{\mathcal{R}}_{j\rho})\\
    =&-\lim\limits_{R\to\infty}
    \Big[\int_{\partial B_R(0)}\f{\partial\tilde{\mathcal{R}}_{j\rho}}{\partial\nu}
    (x\cdot\nabla\tilde{\mathcal{R}}_{j\rho})dS
    -\int_{B_R(0)}\nabla\tilde{\mathcal{R}}_{j\rho}
    \nabla(x\cdot\nabla\tilde{R}_{j\rho})\Big]\\
    =&-\lim\limits_{R\to\infty}
    \Big[\int_{\partial B_R(0)}\f{\partial\tilde{\mathcal{R}}_{j\rho}}{\partial\nu}
    (x\cdot\nabla\tilde{\mathcal{R}}_{j\rho})dS
    -\f{1}{2}\int_{\partial B_R(0)}(x\cdot\nu)|\nabla\tilde{\mathcal{R}}_{j\rho}|^2dS\Big]\\
    =&0,
    \end{align*}
    \begin{align*}
    B_{j\rho}:=&\int_{\R^2}\Big[\ep^2_\rho\mu_{j\rho}
    -\f{\ep^4_\rho\Om^2|x|^2}{4}-\ep^2_\rho V_\Om(\ep_\rho(x+y_0))\Big]\tilde{\mathcal{R}}_{j\rho}
    (x\cdot\nabla\tilde{\mathcal{R}}_{j\rho})\\
    =&-\int_{\R^2}\tilde{\mathcal{R}}^2_{j\rho}
    \Big[\ep^2_\rho\mu_{j\rho}
    -\f{\ep^4_\rho\Om^2|x|^2}{2}-\ep^2_\rho V_\Om(\ep_\rho(x+y_0))-\f{\ep^2_\rho}{2}x\cdot\nabla_x V_\Om(\ep_\rho(x+y_0))\Big],
    \end{align*}
and
    \begin{align*}
    C_{j\rho}:=&(a^*)^\f{p-1}{2}\int_{\R^2}|\tilde{u}_{j\rho}|^{p-1}\tilde{\mathcal{R}}_{j\rho}
    (x\cdot\nabla\tilde{\mathcal{R}}_{j\rho})\\
    =&\f{(a^*)^\f{p-1}{2}}{2}\int_{\R^2}|\tilde{u}_{j\rho}|^{p-1}
    \big[(x\cdot\nabla|\tilde{u}_{j\rho}|^2)
    -(x\cdot\nabla|\tilde{\mathcal{I}}_{j\rho}|^2)\big]\\
    =&\f{(a^*)^\f{p-1}{2}}{p+1}\int_{\R^2}
    (x\cdot\nabla|\tilde{u}_{j\rho}|^{p+1})
    -\f{(a^*)^\f{p-1}{2}}{2}\int_{\R^2}
    |\tilde{u}_{j\rho}|^{p-1}(x\cdot\nabla\tilde{\mathcal{I}}^{2}_{j\rho})\\
    =&-\f{2(a^*)^\f{p-1}{2}}{p+1}\int_{\R^2}
    |\tilde{u}_{j\rho}|^{p+1}-\f{(a^*)^\f{p-1}{2}}{2}\int_{\R^2}
    |\tilde{u}_{j\rho}|^{p-1}(x\cdot\nabla\tilde{\mathcal{I}}^{2}_{j\rho}),
    \end{align*}
Denote
    \begin{equation*}
    D_{j\rho}:=-\int_{\R^2}\ep^2_\rho\Om(x^\perp\cdot\nabla \tilde{\mathcal{I}}_{j\rho})
    (x\cdot\nabla\tilde{\mathcal{R}}_{j\rho}),\,\ j=1,2.
    \end{equation*}
Thus, we derive from (\ref{eqS4.95}) and above that
    \begin{equation}\label{eqS4.96}
    \f{D_{2\rho}-D_{1\rho}}{\|\tilde{u}_{2\rho}-\tilde{u}_{1\rho}\|_{L^\infty(\R^2)}}
    =\f{(B_{2\rho}-B_{1\rho})+(C_{2\rho}-C_{1\rho})}{\|\tilde{u}_{2\rho}-\tilde{u}_{1\rho}\|_{L^\infty(\R^2)}}.
    \end{equation}

We next estimate all terms of \eqref{eqS4.96} as follows. Applying Lemmas \ref{lem 4.4},  \ref{lem 4.3} and \ref{lem 4.6}, we obtain that as $\rho\to\infty$,
    \begin{equation}\label{eqS4.98}
    \begin{split}
    &\f{D_{2\rho}-D_{1\rho}}{\|\tilde{u}_{2\rho}-\tilde{u}_{1\rho}\|_{L^\infty(\R^2)}}\\
    =&-\ep^2_\rho\int_{\R^2}
    \Om (x^\perp\cdot\nabla \tilde{\eta}_{2\rho})
    (x\cdot\nabla\tilde{\mathcal{R}}_{2\rho})
    -\ep^2_\rho\int_{\R^2}
    \Om(x^\perp\cdot\nabla \tilde{\mathcal{I}}_{1\rho})
    (x\cdot\nabla\tilde{\eta}_{1\rho})\\
    =&-\ep^2_\rho\int_{\R^2}
    \Om (x^\perp\cdot\nabla \tilde{\eta}_{2\rho})
    \Big[x\cdot\Big(\nabla\f{w}{\sqrt{a^*}}\Big)
    +x\cdot\nabla\Big(\tilde{\mathcal{R}}_{2\rho}-\f{w}{\sqrt{a^*}}\Big)\Big]+o(\ep^4_\rho)\\
    =&\ep^2_\rho\f{\Om}{\sqrt{a^*}}\int_{\R^2}
    \Big[x^\perp\cdot\nabla(x\cdot\nabla w)\Big]
    \tilde{\eta}_{2\rho}+o(\ep^4_\rho)=o(\ep^4_\rho).
    \end{split}
    \end{equation}
Similarly, by Proposition \ref{prop4.1}, Lemmas \ref{lem 4.3} and \ref{lem 4.6}, we deduce that  as $\rho\to\infty$,
    \begin{align}\label{eqS4.99}
    &\f{C_{2\rho}-C_{1\rho}}{\|\tilde{u}_{2\rho}-\tilde{u}_{1\rho}\|_{L^\infty(\R^2)}}\nonumber\\
    =&-\f{2(a^*)^\f{p-1}{2}}{p+1}\int_{\R^2}
    \f{|\tilde{u}_{2\rho}|^{p+1}-|\tilde{u}_{1\rho}|^{p+1}}{\|\tilde{u}_{2\rho}-\tilde{u}_{1\rho}\|_{L^\infty(\R^2)}}
    -\f{(a^*)^{\f{p-1}{2}}}{2}\int_{\R^2}
    |\tilde{u}_{2\rho}|^{p-1}
    \Big[x\cdot\nabla\big[\tilde{\eta}_{2\rho}(\tilde{\mathcal{I}}_{2\rho}+\tilde{\mathcal{I}}_{1\rho})\big]\Big]\nonumber\\
    &-\f{(p-1)(a^*)^{\f{p-1}{2}}}{4}\int_{\R^2}
     \Big[\tilde{\eta}_{1\rho}(\tilde{\mathcal{R}}_{2\rho}+\tilde{\mathcal{R}}_{1\rho})
      +\tilde{\eta}_{2\rho}(\tilde{\mathcal{I}}_{2\rho}+\tilde{\mathcal{I}}_{1\rho})\Big]\nonumber\\
      &\quad\quad\cdot\int^{1}_{0}\Big[t|\tilde{u}_{2\rho}|^2+(1-t)|\tilde{u}_{1\rho}|^2\Big]^{\f{p-3}{2}}dt
      (x\cdot\nabla\tilde{\mathcal{I}}^2_{1\rho})
    \nonumber\\
    =&-\f{2(a^*)^\f{p-1}{2}}{p+1}\int_{\R^2}
    \f{|\tilde{u}_{2\rho}|^{p+1}-|\tilde{u}_{1\rho}|^{p+1}}{\|\tilde{u}_{2\rho}-\tilde{u}_{1\rho}\|_{L^\infty(\R^2)}}
    +o(\ep^4_\rho).
    \end{align}
As for the term containing $B_{j\rho}$, we obtain from the assumption $(V_2)$ that as $\rho\to\infty$,
    \begin{align*}
    &\f{B_{2\rho}-B_{1\rho}}{\|\tilde{u}_{2\rho}-\tilde{u}_{1\rho}\|_{L^\infty(\R^2)}}\\
    =&\int_{\R^2}
    \Big[\f{\ep^4_\rho\Om^2|x|^2}{2}+\ep^2_\rho V_\Om\big(\ep_\rho(x+y_0)\big)+\f{\ep^2_\rho}{2}(x+y_0)\cdot\nabla_x V_\Om\big(\ep_\rho(x+y_0)\big)\Big]\\
    &\quad\quad\cdot(\tilde{\mathcal{R}}_{2\rho}+\tilde{\mathcal{R}}_{1\rho})
    \tilde{\eta}_{1\rho}dx-J_\rho-K_\rho\\
    =&\int_{\R^2}
    \Big[\f{\ep^4_\rho\Om^2|x|^2}{2}+\f{2+s}{2}\ep^{2+s}_\rho h(x+y_0)\Big]
    (\tilde{\mathcal{R}}_{2\rho}+\tilde{\mathcal{R}}_{1\rho})\tilde{\eta}_{1\rho}dx
    -J_\rho-K_\rho+o(\ep^{2+s}_\rho)\\
    =&O(\ep^{2+s}_\rho)
    -J_\rho-K_\rho,
    \end{align*}
where the fact that $x\cdot\nabla h(x)=sh(x)$ is used in the last equality. Here the terms $J_\rho$ and $K_\rho$ are defined as
$$J_\rho:=\f{\ep^2_\rho}{2}\int_{\R^2}
\Big[y_0\cdot\nabla_x V_\Om\big(\ep_\rho(x+y_0)\Big]
(\tilde{\mathcal{R}}_{2\rho}+\tilde{\mathcal{R}}_{1\rho})\tilde{\eta}_{1\rho}dx,$$
and
$$K_\rho:=\ep^2_\rho\int_{\R^2}\f{\tilde{\mathcal{R}}^2_{2\rho}\mu_{2\rho}
-\tilde{\mathcal{R}}^2_{1\rho}\mu_{1\rho}}
{\|\tilde{u}_{2\rho}-\tilde{u}_{1\rho}\|_{L^\infty(\R^2)}}dx.$$
Using \eqref{A14concentrate.rate}, we note from the first identity of \eqref{eqS4.94} that as $\rho\to\infty$,
$$J_\rho=o(\ep^4_\rho).$$
Applying Lemmas \ref{lem 4.5}, \ref{lem 4.3} and \ref{lem 4.6}, we obtain from \eqref{S2} that as $\rho\to\infty$,
    \begin{align*}
    K_\rho:=&\ep^2_\rho\int_{\R^2}
    \f{|\tilde{u}_{2\rho}|^2\mu_{2\rho}
    -|\tilde{u}_{1\rho}|^2\mu_{1\rho}}
     {\|\tilde{u}_{2\rho}-\tilde{u}_{1\rho}\|_{L^\infty(\R^2)}}
    -\ep^2_\rho\int_{\R^2}\f{\tilde{\mathcal{I}}^2_{2\rho}\mu_{2\rho}
    -\tilde{\mathcal{I}}^2_{1\rho}\mu_{1\rho}}
    {\|\tilde{u}_{2\rho}-\tilde{u}_{1\rho}\|_{L^\infty(\R^2)}}\\
    =&\f{\ep^2_\rho(\mu_{2\rho}-\mu_{1\rho})}
     {\|\tilde{u}_{2\rho}-\tilde{u}_{1\rho}\|_{L^\infty(\R^2)}}
    -\ep^2_\rho\int_{\R^2}\tilde{\eta}_{2\rho}
    (\tilde{\mathcal{I}}_{2\rho}+\tilde{\mathcal{I}}_{1\rho})\mu_{2\rho}\\
    &-\f{\ep^2_\rho(\mu_{2\rho}-\mu_{1\rho})}
     {\|\tilde{u}_{2\rho}-\tilde{u}_{1\rho}\|_{L^\infty(\R^2)}}
    \int_{\R^2}\tilde{\mathcal{I}}^2_{1\rho}\\
    =&-\f{(p-1)(a^*)^\f{p-1}{2}}{p+1}\int_{\R^2}
    \f{|\tilde{u}_{2\rho}|^{p+1}-|\tilde{u}_{1\rho}|^{p+1}}
    {\|\tilde{u}_{2\rho}-\tilde{u}_{1\rho}\|_{L^\infty(\R^2)}}+o(\ep^4_\rho).
    \end{align*}
It then follows from above that as $\rho\to\infty$,
    \begin{equation}\label{B}
    \begin{split}
    &\f{B_{2\rho}-B_{1\rho}}{\|\tilde{u}_{2\rho}-\tilde{u}_{1\rho}\|_{L^\infty(\R^2)}}\\
    =&O(\ep^{2+s}_\rho)
    +\f{(p-1)(a^*)^\f{p-1}{2}}{p+1}\int_{\R^2}
    \f{|\tilde{u}_{2\rho}|^{p+1}-|\tilde{u}_{1\rho}|^{p+1}}
    {\|\tilde{u}_{2\rho}-\tilde{u}_{1\rho}\|_{L^\infty(\R^2)}}.
    \end{split}
    \end{equation}
From (\ref{eqS4.96})-(\ref{B}), we deduce that as $\rho\to\infty$,
    \begin{align}\label{eqS4.101}
    O(\ep^{2+s}_\rho)
    =&\f{(3-p)(a^*)^{\f{p-1}{2}}}{p+1}\int_{\R^2}
    \f{|\tilde{u}_{2\rho}|^{p+1}-|\tilde{u}_{1\rho}|^{p+1}}
    {\|\tilde{u}_{2\rho}-\tilde{u}_{1\rho}\|_{L^\infty(\R^2)}}\nonumber\\ =&\f{(3-p)(a^*)^{\f{p-1}{2}}}{4}\int_{\R^2}
    \Big\{\big(|\tilde{u}_{2\rho}|^{\f{p+1}{2}}+|\tilde{u}_{1\rho}|^{\f{p+1}{2}}\big)
    \Big[(\tilde{\mathcal{R}}_{2\rho}+\tilde{\mathcal{R}}_{1\rho})\tilde{\eta}_{1\rho}
    +(\tilde{\mathcal{I}}_{2\rho}+\tilde{\mathcal{I}}_{1\rho})\tilde{\eta}_{2\rho}\Big]\nonumber\\
    & \ \ \ \ \ \ \ \ \ \ \ \ \ \ \ \ \ \ \ \ \ \ \ \ \ \ \ \ \ \ \
    \cdot\int^1_0\Big[t|\tilde{u}_{2\rho}|^2+(1-t)|\tilde{u}_{1\rho}|^2\Big]^{\f{p-3}{4}}dt\Big\}dx\nonumber\\
    =&\f{(3-p)}{\sqrt{a^*}}\int_{\R^2}
    w^p\tilde{\eta}_1dx+o(1).
    \end{align}
Using Lemma \ref{lem 4.5}, it then follows from (\ref{eqS4.101}) that
    \begin{align*}
    0=&\int_{\R^2}w^p\tilde{\eta}_1\\
    =&\int_{\R^2}w^p\Big[b_0\big(w+\f{p-1}{2}x\cdot\nabla w\big)
     +\sum^2_{j=1}b_j\f{\partial w}{\partial x_j}\Big]\\
    =&b_0\int_{\R^2}w^{p+1}
    +\f{b_0}{2}\f{p-1}{p+1}\int_{\R^2}x\cdot\nabla w^{p+1}
    +\sum^2_{j=1}\f{b_j}{p+1}\int_{\R^2}\f{\partial w^{p+1}}{\partial x_j}\\
    =&b_0\int_{\R^2}w^{p+1}
    -b_0\f{p-1}{p+1}\int_{\R^2}w^{p+1}\\
    =&b_0\Big[1-\f{p-1}{p+1}\Big]\int_{\R^2}w^{p+1},
    \end{align*}
which implies $b_0=0$ due to $1<p<3$.
\vskip0.2cm
\noindent $\mathbf{Step\,\ 3}$. The constants $b_1=b_2=0$.

By step 2, we derive from (\ref{eqS4.73}) that
     \begin{equation*}
     \sum^2_{j=1}b_j\int_{\R^2}\f{\partial h(x+y_0)}{\partial x_l}\f{\partial w^2}{\partial x_j}dx
     =0, \,\ l=1,2,
     \end{equation*}
which implies from the non-degeneracy assumption of $H(y)$ in \eqref{nondege} that $b_1=b_2=0$, and the proof of Step 3 is thus completed.

\vskip 0.1truein
Since $\|\tilde{\eta}_\rho\|_{L^\infty(\R^2)}=1$, we can deduce from the exponential decay of Lemma \ref{lem 4.4} that  $\tilde{\eta}_\rho\to\tilde{\eta}_0=\tilde{\eta}_1+i\tilde{\eta}_2\not\equiv0$ uniformly in $C^1(\R^2)$ as $\rho\to\infty$. However, Steps 2 and  3 imply that $\tilde{\eta}_0\equiv0$, this is a contraction. Therefore, we complete the proof of Theorem \ref{thm.localuniqueness}.

\qed

\appendix
\section{Appendix}
\renewcommand{\theequation}{A.\arabic{equation}}
\setcounter{equation}{0}
In the appendix, we shall prove the equivalence between ground states of equation \eqref{bc1-1} and minimizers of \eqref{A1cvp}. We first introduce the definition of ground states of \eqref{bc1-1}.
Given any $\rho\in(0,\infty)$ and $0<\Om<\Om^*$, the energy functional of \eqref{bc1-1} is defined by
\begin{equation}\label{bc-3}
\begin{split}
F_{\mu,\rho}(u):&=\int_{\R^2}\big[|\nabla u|^{2}+(V(x)-\mu)|u|^{2}\big]dx-\f{2\rho^{p-1}}{p+1}\int_{\R^2}|u|^{p+1}dx\\
&\quad\quad-\Om\int_{\R^2}x^\perp\cdot(iu,\nabla u)dx,
\end{split}
\end{equation}
where $\mu\in\R$ is a parameter and the energy functional $E_\rho(u)$ is given by \eqref{A2ef}.
Define
$$S_{\mu, \rho}:=\{u\in \mathcal{H}\setminus \{0\}:\langle F^{'}_{\mu, \rho}(u), \varphi \rangle=0\,\ \text{for all}\,\ \varphi\in \mathcal{H}\}, $$
where
\begin{equation*}
\begin{split}
\langle F^{'}_{\mu, \rho}(u), \varphi \rangle
&=2Re\Big\{\int_{\R^2}\big[\nabla u\nabla \bar{\varphi}+(V(x)-\mu)u\bar{\varphi}\big]dx-\rho^{p-1}\int_{\R^2}|u|^{p-1}u\bar{\varphi}dx\\
&\quad\quad\quad\quad+\int_{\R^2}i\Om(x^\perp\cdot\nabla u)\bar{\varphi}dx\Big\},
\end{split}
\end{equation*}
and
\begin{equation}\label{bc-4}
G_{\mu, \rho}:=\{u\in S_{\mu, \rho}:F_{\mu, \rho}(u)\leq F_{\mu, \rho}(v), \ \text{for all} \ v\in S_{\mu, \rho}\}.
\end{equation}
If $u\in G_{\mu, \rho}$, we say that $u$ is a {\em ground state} of \eqref{bc1-1}. Now we give the following theorem on the equivalence between ground states of equation \eqref{bc1-1} and minimizers of \eqref{A1cvp}.
\begin{thm}\label{bc}
Suppose $\rho\in(0,\infty)$ and $0<\Om<\Om^*$ are given, then any minimizer of \eqref{A1cvp} is a ground state of \eqref{bc1-1} for some $\mu\in\R$; conversely, any ground state of \eqref{bc1-1} for some $\mu\in\R$ is a minimizer of \eqref{A1cvp}.
\end{thm}
Since the proof of Theorem A.1 is similar to \cite [Proposition A.1] {GLWZ}, 
we omit it here.
	\qed


\end{document}